\documentclass{amsart}
\usepackage[utf8]{inputenc}
\usepackage{amsfonts}
\usepackage{amsmath}
\usepackage{amssymb}
\usepackage{amsthm}
\usepackage{caption}
\usepackage{cite}
\usepackage{dutchcal}
\PassOptionsToPackage{hyphens}{url}\usepackage[colorlinks=true,urlcolor=black,citecolor={black},linkcolor = black]{hyperref}
\usepackage[dvipsnames]{xcolor}
\usepackage{enumerate}
\usepackage{extarrows}
\usepackage{geometry}
\usepackage{mathdots}
\usepackage{mathrsfs}
\usepackage{setspace}
\usepackage{subcaption}
\usepackage{tikz}
\usepackage{tikz-cd}
\usepackage{todonotes}
\usepackage{verbatim}

\theoremstyle{definition}
\newtheorem{theorem}{Theorem}[section]
\newtheorem{lemma}[theorem]{Lemma}

\newtheorem{proposition}[theorem]{Proposition}

\newtheorem{definition}[theorem]{Definition}
\newtheorem{example}[theorem]{Example}

\newcommand{\End}{\text{End}}
\newcommand{\OO}{\mathcal{O}}
\newcommand{\Aut}{\mathrm{Aut}}

\newcommand{\ZZ}{\mathbb{Z}}

\newcommand{\Nrd}{{\operatorname{Nrd}}}
\newcommand{\ec}[1]{(#1)}

\renewcommand{\k}{\kappa}

\renewcommand{\hat}[1]{\widehat{#1}}

\newcommand{\QQ}{\mathbb{Q}}

\counterwithin{figure}{section}

\usepackage{subfiles}

\newcommand{\notinsubfile}[1]{}

\newcommand{\Fp}{{\mathbb{F}_p}}
\newcommand{\Fpp}{{\mathbb{F}_{p^2}}}
\newcommand{\Fpbar}{{\overline{\mathbb{F}}_{p}}}

\newcommand{\FpbarGraph}{{\mathcal{G}^\ell_{\Fpbar}}}

\newcommand{\EGSet}{{|\mathcal{S}_N|}}
\newcommand{\EGCat}{{\mathcal{S}_N}}

\newcommand{\Hom}{\text{Hom}}

\newcommand{\EpCat}{{\mathcal{E}^N_{p,\ell}}}

\newcommand{\hh}[1]{\mathcal{h}_{#1}}

\usepackage{stmaryrd}

\title[Supersingular Curves with Level Structure]{Adding Level Structure to Supersingular Elliptic Curve Isogeny Graphs}
\author{Sarah Arpin}
\address{Mathematics Institute \\
Universiteit Leiden\\
Leiden, The Netherlands}
\email{sarpinmath@gmail.com}
\date{\today}


\thanks{The author was supported by NSF-CAREER CNS-1652238 (P.I. Katherine E. Stange)}
\begin{document}

\maketitle

%\tableofcontents

\begin{center}
\begin{abstract}
In this paper, we add the information of level structure to supersingular elliptic curves and study these objects with the motivation of isogeny-based cryptography. Supersingular elliptic curves with level structure map to Eichler orders in a quaternion algebra, just as supersingular elliptic curves map to maximal orders in a quaternion algebra via the classical Deuring correspondence. We study this map and the Eichler orders themselves. We also look at isogeny graphs of supersingular elliptic curves with level structure, and how they relate to graphs of Eichler orders.
\end{abstract}
\end{center}

\section{Introduction}

Supersingular elliptic curve isogeny graphs have a rich underlying mathematical structure, and yet they appear to be difficult to navigate in a computational sense. The recent appearance of these graphs in cryptographic protocols which aim to be safe against classical and quantum attacks has led to a resurgence of interest in the mathematical properties of supersingular elliptic curves. In this work, we study supersingular elliptic curves endowed with level-$N$ structure:
\begin{definition}[Supersingular elliptic curve with level-$N$ structure, see Definition~\ref{def:SSECsWithLS}]\label{def:SSECsWithLS_intro}
Let $p$ be a prime and $N$ an integer coprime to $p$.
Let $\EGSet$ denote the set of pairs $(E,G)$, up to isomorphism, where $E$ is a supersingular elliptic curve over $\Fpbar$ and $G\subseteq E[N]$ is a cyclic subgroup of order $N$.
\end{definition}

To better understand supersingular elliptic curves with level-$N$ structure, we study maps on these curves, in particular their endomorphisms: 
\begin{definition}[Endomorphisms of $(E,G)$, see Definition~\ref{def:OO(E,G)}]\label{def:OO(E,G)_intro}
As a subring of $\End(E)$, we define the ring of endomorphisms of the pair $(E,G)\in\EGSet$ as follows:
\[\mathcal{O}(E,G) := \{\alpha\in \text{End}(E): \alpha(G)\subseteq G\}.\]
\end{definition}

The endomorphism rings of supersingular elliptic curves are maximal orders in a quaternion algebra. This correspondence is an explicit equivalence of categories called the Deuring correspondence \cite{Deu41}.  Research into this correspondence has greatly expanded our understanding of supersingular elliptic curves, and in this work we extend these tools to supersingular elliptic curves with level-$N$ structure. In Theorem~\ref{thm:O(E,G)_is_EO}, restated below, we prove that the endomorphism rings $\OO(E,G)$ are Eichler orders of level-$N$ in the quaternion algebra $\End(E)\otimes_\ZZ\QQ$. 
\begin{theorem}[See Theorem~\ref{thm:O(E,G)_is_EO}]\label{thm:O(E,G)_is_EO_intro}
$\OO(E,G)$ is isomorphic to an Eichler order of level $|G| = N$. 
\end{theorem}

The correspondence between isomorphism classes of supersingular elliptic curve endomorphism rings and isomorphism classes of maximal orders in a quaternion algebra is not a priori an injective map. The failure of injectivity comes from curves which are related by the $p$-power Frobenius map, as these curves have isomorphic endomorphism rings. In the case of supersingular elliptic curves with level-$N$ structure the failure of injectivity is more complicated, but also reveals more information about the structure of the supersingular elliptic curves. In this case, we restrict to $N$ squarefree and we study the failure of injectivity via involutions which we define on the isomorphism classes in $\EGSet$. This complete description can be found in Section~\ref{sec:vertices_to_EOs}, culminating in Theorem~\ref{thm:fiber_aboveO(E,G)} which completely describes the failure of injectivity of the map from $\EGSet$ to isomorphism classes of Eichler orders of level-$N$. In addition to understanding the fiber size of the map taking supersingular elliptic curves with level-$N$ structure to their Eichler order endomorphism rings, we also understand the structure on the quaternion side which dictates this fiber size. We restate a version of Theorem~\ref{thm:fiber_aboveO(E,G)} below:

\begin{theorem}
Fix an Eichler order $\mathcal{O}$ of level $N = q_1\cdots q_r$ squarefree and coprime to $p$. The number of isomorphism classes of supersingular elliptic curves with level-$N$ structure $(E,G)\in\EGSet$ with endomorphism ring isomorphic to $\mathcal{O}$ is equal to the size of the two-sided ideal class group of $\mathcal{O}$, which equals $2^k$ for some $k\in\{0,1,...,r+1\}$.
\end{theorem}

With all of this structure understood, we prove a formal equivalence of categories for supersingular elliptic curves with level-$N$ structure, in the style of the classical Deuring correspondence: 
\begin{theorem}[Equivalence of Categories, see Theorem~\ref{thm:cats_equiv}]\label{thm:cats_equiv_intro}
Fix a supersingular elliptic curve $E$ defined over $\overline{\mathbb{F}}_p$ and a cyclic subgroup $G\subset E[N]$ of squarefree order $N$, coprime to $p$. There is a contravariant functor $\hh{(E,G)}$ from the category $\mathcal{S}_N$ of supersingular elliptic curves with level-$N$ structure to the category $\mathcal{LM}$ of invertible left $\mathcal{O}(E,G)$-modules. This functor defines an equivalence of categories.
\end{theorem}

Finally, we consider maps between supersingular elliptic curves with level-$N$ structure in the form of $\ell$-isogeny graphs, which are defined and explored in Section~\ref{sec:Graph}. We briefly recap some highlights of isogeny-based cryptography as motivation for studying this isogeny graph variant. Supersingular elliptic curve $\ell$-isogeny graphs were first proposed for use in post-quantum cryptography in 2006 by Charles, Goren, and Lauter \cite{CGL06} with a hash function based on walks in the $\ell$-isogeny graph. This was swiftly followed by papers of Rostovtsev-Stolbunov and Couveignes whose work (which was actually from 1997, but not made public at that time) took the perspective of using a class group action to walk the $\ell$-isogeny graph. These works formed the basis for CSIDH \cite{CSIDH}, a key exchange protocol based on isogenies of supersingular elliptic curves over $\Fp$. Perhaps the most well-known isogeny-based cryptographic protocol was SIKE \cite{SIKE} (a variant of SIDH), a key exchange protocol which made public the images of certain torsion subgroup generators under certain isogenies. 
In 2022, Castryck-Decru \cite{CastryckDecru}, Maino-Martindale \cite{MainoMartindale,MainoMartindale2}, and Robert \cite{Robert} were able to use this information to break the protocol and SIDH in general. This break has shaken trust in isogeny-based cryptography, but the reliance on the extra torsion point information was key to the break; this highlights the importance of understanding isogeny graphs which are enhanced by additional structure. Another protocol of mathematical interest is OSIDH \cite{OSIDH}, in which the supersingular elliptic curves are endowed with the information of an endomorphism in the form of an ``orientation". These graphs have been studied extensively \cite{WIN5I, WIN5II, Bweso_orientations}, and in future work we hope to explore the connection between orientations and level structure.

\subsection{Historical context of level structure}

The notion of extending the Deuring correspondence to a supersingular $\ell$-isogeny graph with level structure is not new. However, there has yet to appear a detailed description of an equivalence of categories for supersingular elliptic curves with level structure. The idea has been called ``folklore" \cite[Section 4]{SqiSign}: Papers have been written about related concepts in the context of modular forms (Ribet), or about different choices of level structure (Goren--Kassaei \cite{goren_kassaei}, with a choice of torsion point, Roda's thesis \cite{RodaThesis} with full level structure). In this paper, the author hopes to provide the details of theorems that many have suspected, as well as some which are perhaps less expected. To begin this work, we provide a brief overview of what we have found in the literature to date.

Voight \cite[Remark 42.3.10]{QASVOIGHT} notes that a generalization of the Deuring correspondence is possible through mild adjustments.
Ribet, in \cite{Ribet1989}, also notes that the Deuring correspondence as phrased by Mestre--Osterl\'e can be extended to ``oriented" Eichler orders, but does not prove the correspondence explicitly.

Eichler \cite{Eichler01} \cite{Eichler2} and Pizer \cite{PizerTypeNumbs} provide the foundational theory of Eichler orders. 

More recently, work of Goren and Kassaei \cite{goren_kassaei} takes the perspective of Hecke operators to prove properties of the supersingular $\ell$-isogeny graph with the added level-$N$ structure of a choice of $N$-torsion point.

The SqiSign authors \cite{SqiSign} have most recently published a version of the Eichler order Deuring correspondence, motivated by commutative isogeny diagrams of supersingular elliptic curves: Under suitable conditions for $p$, the authors prove a bijection between the class set of a fixed Eichler order of squarefree level $N$ and the set of all $N$-isogenies between supersingular elliptic curves over $\Fpbar$. This bijection is essentially the same as the underlying bijection on objects of the equivalence of categories proved in Section~\ref{sec:category}.

Since the first appearance of this paper, Codogni and Lido \cite{CodogniLido} have continued to understand the spectral properties of isogeny graphs with level structure. Page and Wesolowski \cite{PageWeso} introduced a framework to study a generalized notion of level structure.

\subsection{Acknowledgments}
The author is deeply indebted to her advisor, Katherine E. Stange, for continuous guidance on this paper from the very start. The author would like to thank John Voight for promptly answering emails to provide very helpful exposition and clarification. Many thanks to David Grant and Christelle Vincent for their feedback on the thesis-version of this work. Additional thanks to Leo Herr, Soumya Sankar, and Jaap Top for helpful discussions. Many thanks to the reviewer for their encouraging feedback and helpful suggestions, especially in Sections~\ref{sec:vertices_to_EOs} and \ref{sec:category}.

\subsection{Conventions}
In this paper, $p$ is a (cryptographic size) prime, $\ell$ is a small prime, and $N$ is an integer which is coprime to $p\ell$. In Sections~\ref{sec:vertices_to_EOs} and \ref{sec:category}, we restrict $N$ to be squarefree. In Section~\ref{sec:countingNisogconjs}, we require $N$ to be prime.

\section{Background}\label{sec:background}
Let $B_{p,\infty}$ denote the unique (up to isomorphism) quaternion algebra ramified precisely at $p$ and infinity.

\subsection{The Classical Deuring Correspondence}\label{ssec:bckgrnd_classicDeuring}
Deuring provides a correspondence between the endomorphism rings of supersingular elliptic curves over $\Fpbar$ and maximal orders in the appropriate quaternion algebra. The connection to the quaternions provides an important avenue for studying the structure of the isogeny graphs.

\begin{theorem}[Deuring Correspondence]
Fix a maximal order $R$ of the quaternion algebra $B_{p,\infty}$. There is a bijection between isomorphism classes of supersingular elliptic curves over $\Fpbar$ and the left class set of the maximal order $R$. 
\end{theorem}

Deuring's original statement depends on a choice of maximal order $R$ in $B_{p,\infty}$, which is implicitly a choice of supersingular elliptic curve whose endomorphism ring is isomorphic to $R$. For every maximal order, the right orders of ideals in the left ideal class set of that order will run through all of the maximal orders of the quaternion algebra. 
In this way, one can think of mapping the supersingular elliptic curves over $\Fpbar$ to the maximal orders of $B_{p,\infty}$. The fibers of this map have either one or two elements, depending on the field of definition of the supersingular elliptic curve, or equivalently the size of the two-sided ideal class group of the maximal order. 
This perspective removes the necessity of an initial choice of maximal order, but it no longer describes a bijection: If $E$ is defined over $\Fpp\setminus\Fp$, then $E\not\cong E^p$, but $\End(E)\cong\End(E^p)$ map to the same maximal order of $B_{p,\infty}$. If $E$ is defined over $\Fp$, then $E\cong E^p$ and $\End(E)$ is the maximal order uniquely identified with the isomorphism class of $E$. 

Ribet \cite{Ribet1989} credits an unpublished manuscript of Mestre--Osterl\'e for this basepoint-free version of the Deuring Correspondence: He writes that Mestre--Osterl\'e take a perspective of ``oriented" maximal orders to achieve this result. The basepoint-free perspective is also how Kohel presents the Deuring correspondence in his thesis \cite{Kohel}:

\begin{theorem}[Theorem 44 \cite{Kohel}]\label{thm:bp-free-Deuring}
Given a maximal order of the quaternion algebra $B_{p,\infty}$, there exist one or two supersingular $j$-invariants over $\Fpbar$ such that the corresponding endomorphism ring is isomorphic to a maximal order of the given type.
\end{theorem}

Kohel also presents the basepoint dependent version of the Deuring Correspondence as a categorical equivalence \cite[Theorem 45]{Kohel}. In Section~\ref{sec:category}, we prove a categorical equivalence in the level structure context.

\subsection{Translating Isogenies to the Quaternion Algebra Side}\label{ssec:bckgrnd_isogenies}
Isogenies of supersingular elliptic curves also have a corresponding object in the quaternion algebra $B_{p,\infty}$. A thorough reference for the correspondence between isogenies and left ideals of a maximal order $\OO\cong \End(E)$ is described in detail in \cite[Section 42.2]{QASVOIGHT}. We briefly recall this theory here: suppose $\varphi:E\to E'$ is a separable isogeny between supersingular elliptic curves over $\Fpbar$. Let $I_\varphi$ be the left ideal of $\End(E)$ in $B_{p,\infty}$ which corresponds to $\varphi$ in the following way: 
\[\ker(\varphi) = \bigcap_{\alpha\in I_\varphi}\ker(\alpha).\]
The norm of $I_\varphi$ is equal to the degree of $\varphi$. The ideal $I_\varphi$ is also a right $\End(E')$ ideal, by the same theory.

\subsection{Embedding Multiple Endomorphism Rings in \texorpdfstring{$B_{p,\infty}$}{}}\label{ssec:bckgrnd_multipleendosembedding}
If one wishes to compare more than one supersingular elliptic curve over $\Fpbar$ to the corresponding maximal order in $B_{p,\infty}$, one must be careful to choose compatible maps into the same copy of $B_{p,\infty}$. A detailed discussion is found in \cite[Section 42.2]{QASVOIGHT}, and we provide a summary of the details which will be necessary for this paper. Fix a supersingular elliptic curve $E/\Fpbar$. The  endomorphism ring of $E$, $\End(E)$, is a maximal order in the quaternion algebra $B_{p,\infty}^E:=\End(E)\otimes_\ZZ\QQ\cong B_{p,\infty}$. All supersingular elliptic curves are isogenous. To map the endomorphism ring $\End(E')$ of another supersingular elliptic curve $E'/\Fpbar$ into $B_{p,\infty}^E$, we choose an isogeny $\varphi:E\to E'$. As described above, $\varphi$ corresponds to a left ideal $I$ of the maximal order $\End(E)$. Any left ideal in the class of $I$ corresponds to an isogeny $E\to E'$. We map the endomorphisms of $E'$ into $B_{p,\infty}^E$ via 
\begin{equation}\label{eq:EndoRingToMaxOrder}
\begin{split}
    \End(E')&\hookrightarrow B_{p,\infty}^E\\
    \beta &\mapsto \frac{1}{\deg\varphi}(\hat{\varphi}\beta\varphi).
\end{split}
\end{equation}
The image of $\End(E')$ is the maximal order of $B_{p,\infty}^E$ which is the right order of $I$. In this way, we are viewing the endomorphism rings of $E$ and $E'$ inside the same copy of $B_{p,\infty}$, namely $B_{p,\infty}^E$ as defined above. Note that this map depends on a choice of $\varphi$. If instead we had chosen an isogeny $\varphi':=\eta\circ\varphi:E\to E'$, where $\eta\in\Aut(E')$, the image of $\End(E')$ in $B_{p,\infty}^E$ would remain the same, but the map itself would be different.

\subsection{\texorpdfstring{$\Fp$}{}-Endomorphism Rings}\label{ssec:bckgrn_FpEndoRing}
While computing the full endomorphism ring of a given supersingular elliptic curve is generically a hard problem, this is not the case for computing the subset of endomorphisms which are defined over $\Fp$, for curves which are defined over $\Fp$. Delfs and Galbraith \cite{DelGal01} show that $\End_{\Fp}(E)\cong \ZZ[\sqrt{-p}]$ or $\ZZ[\frac{1 + \sqrt{-p}}{2}]$, depending on the congruence class of $p$ modulo 4 and the action of the $p$-power Frobenius on the two-torsion points of $E$. We condense and re-state this theorem below for ease of reference: 
\begin{proposition}[Section 2{\cite{DelGal01}}]
Let $E/\Fp$ be a supersingular elliptic curve, and let $\pi_p$ denote the $p$-power Frobenius map on $E$. If $p\equiv 1\pmod{4}$, then $\End(E)\cong\ZZ[\sqrt{-p}]$. If $p\equiv 3\pmod{4}$, then there are two possibilities for $\End_{\Fp}(E)$: if $\pi_p(P) = P$ for all $P\in E[2]$, then $\End_{\Fp}(E)\cong\ZZ[\frac{1+ \sqrt{-p}}{2}]$. Otherwise, $\End_\Fp(E)\cong\ZZ[\sqrt{-p}]$.
\end{proposition}

\section{Elliptic Curves with Level Structure and Their Endomorphism Rings}

\begin{definition}\label{def:SSECsWithLS}
Let $p$ be a prime and $N$ an integer coprime to $p$.
Let $\EGSet$ denote the set of pairs $(E,G)$, up to equivalence $\sim$, where $E$ is a supersingular elliptic curve over $\Fpbar$ and $G\subseteq E[N]$ is a cyclic subgroup of order $N$. Two pairs $(E_1,G_1), (E_2,G_2)$ are equivalent under the equivalence relation $\sim$ if there exists an isomorphism $\rho:E_1\to E_2$ such that $\rho(G_1) = G_2$. The pairs in $\EGSet$ are \textit{supersingular elliptic curves with level-$N$ structure}. 
\end{definition}

We define the notion of an endomorphism ring of a pair in $\EGSet$ (Section~\ref{ssec:OO(E,G)}) and describe the structure of this endomorphism ring as an object in the quaternion algebra $B_{p,\infty}$ (Section~\ref{ssec:EOS}).

\subsection{Endomorphism Rings \texorpdfstring{$\OO(E,G)$}{}}\label{ssec:OO(E,G)}

\begin{definition}[$\mathcal{O}(E,G)$]\label{def:OO(E,G)}
As a subring of $\End(E)$, we define the ring of endomorphisms of the pair $(E,G)\in\EGSet$ as follows:
\[\mathcal{O}(E,G) := \{\alpha\in \text{End}(E): \alpha(G)\subseteq G\}.\]
\end{definition}

Since $\EGSet$ is a set of equivalence classes, we need to check that $\OO(\cdot,\cdot)$ is well-defined on these equivalence classes.

\begin{proposition}
Let $(E,G),(F,H)\in\EGSet$ and suppose that there exists an isomorphism $\eta:E\to F$ such that $\eta(G) = H$. Then, the map $\OO(F,H)\to\OO(E,G)$ defined $\alpha\mapsto \eta^{-1}\alpha\eta$ is an isomorphism. 
\end{proposition}
\begin{proof}
If $\eta:E\to F$ is an isomorphism, then we have an isomorphism $\End(F)\to\End(E)$ given by $\alpha\mapsto \eta^{-1}\alpha\eta$. Since $\eta(G) = H$, $\alpha(H)\subseteq H$ is equivalent to $\eta^{-1}\alpha\eta(G)\subseteq G$. We have:
\begin{equation*}
\begin{split}
    \OO(F,H) &= \{\alpha\in\End(F): \alpha(H)\subseteq H\}\\
    &\cong \{\beta\in\End(E):\beta(G)\subseteq G\}\\
    &=\OO(E,G).
\end{split}
\end{equation*}
\end{proof}

In Theorem~\ref{thm:O(E,G)_is_EO}, we show that $\OO(E,G)$ is an Eichler order of level $N$ of $B_{p,\infty}$. We consider $\OO(\cdot,\cdot)$ as a map that we apply to elements $(E,G)$ of $\EGSet$. Just as supersingular elliptic curves are mapped to the set of maximal orders of $B_{p,\infty}$, we map elements of $\EGSet$ to Eichler orders of level $N$ of $B_{p,\infty}$. By Proposition~\ref{prop:Eichler_order_correspond}, the map $\OO(\cdot,\cdot)$ is surjective onto isomorphism classes of Eichler orders of level $N$ in $B_{p,\infty}$, but injectivity fails in an interesting way. We describe this completely in Section~\ref{sec:vertices_to_EOs}.

\begin{proposition}\label{prop:Int_of_two_MaxOrds}
Let $(E,G)$ be an element of $\EGSet$. Let $\varphi:E \to E/G$ be an isogeny with $\ker(\varphi) = G$. Then, $\OO(E,G)= \End(E)\cap(\frac{1}{\deg\varphi}\widehat{\varphi}\End(E/G)\varphi)$ where the intersection is taken in $B_{p,\infty}^E$ via the embedding described in equation~\eqref{eq:EndoRingToMaxOrder} and is independent of choice of $\varphi$.
\end{proposition} 

\begin{proof} 
We proceed by showing containment in both directions. To see $\OO(E,G)\supseteq \End(E)\cap(\frac{1}{\deg\varphi}\widehat{\varphi}\End(E/G)\varphi)$, take $\alpha\in \End(E)\cap(\frac{1}{\deg\varphi}\widehat{\varphi}\End(E/G)\varphi)$. Immediately we have $\alpha\in\End(E)$, so it remains to show $\alpha(G)\subseteq G$. There exists $\beta\in\End(E/G)$ such that $\varphi\circ\alpha = \beta\circ\varphi$. This guarantees that $\alpha(G)\subseteq G$, as $\varphi\circ\alpha(G)=\beta\circ\varphi(G) =\{O_{E/G}\} $.

To see $\OO(E,G)\subseteq \End(E)\cap(\frac{1}{\deg\varphi}\widehat{\varphi}\End(E/G)\varphi)$, take $\alpha\in \OO(E,G)$. To show $\alpha\in \frac{1}{\deg\varphi}\widehat{\varphi}\End(E/G)\varphi$, we will show that there exists a $\beta\in \End(E/G)$ such that $\varphi\circ\alpha = \beta\circ\varphi$. 
Since $\alpha(G)\subseteq G$, we have $\ker(\varphi) = G\subseteq\ker(\varphi\circ\alpha)$. We apply Corollary III.4.11 of \cite{AEC} to guarantee the existence of a (unique) $\beta:E/G\to E/G$ such that $\varphi\circ\alpha = \beta\circ\varphi$.

Our choice of $\varphi:E\to E/G$ is unique up to post-composition with an automorphism of $E/G$. If we replace $\varphi$ above with $\psi := \eta\circ\varphi$ for some $\eta\in\Aut(E/G)$, we obtain the object:
\[\frac{1}{\deg\psi}\widehat{\psi}\End(E/G)\psi = \frac{1}{\deg\eta\cdot\deg\varphi}\widehat{\varphi}\widehat{\eta}\End(E/G)\eta\varphi.\] 
Since $\eta$ is an automorphism of $E/G$, $\widehat{\eta}\End(E/G)\eta = \End(E/G)$ and $\deg\eta = 1$. This gives an equality:
\[\frac{1}{\deg\psi}\widehat{\psi}\End(E/G)\psi = \frac{1}{\deg\varphi}\widehat{\varphi}\End(E/G)\varphi. \]
\end{proof}

\subsubsection{The effect of extra automorphisms on $\EGSet$}\label{ssec:pnot1mod12}
Counting the number of elements of $\EGSet$ with $j(E)=0,1728$ is not as straightforward as counting subgroups of $\ZZ/N\ZZ\times \ZZ/N\ZZ$ of order $N$, due to the presence of extra automorphisms.
The automorphisms $[\pm 1]$ of any $E$ necessarily map $G$ to itself. In the case where $\Aut(E) = \{[\pm1]\}$, this means a fixed Weierstrass equation for a curve $E$ will have $(E,G_0)\sim(E,G_1)$ if and only if $G_0 = G_1$. In particular, every supersingular $j$-invariant not equal to $0$ or $1728$ will have the same number of equivalence classes in $\EGSet$: we refer to this as the generic case.

This is not necessarily the case for the extra automorphisms of $E$ with $j(E) = 0,1728$. The automorphisms of these curves can provide equivalences $(E,G_0)\sim(E,G_1)$ for $G_0\neq G_1$. In these cases, the $j$-invariants $0$ and $1728$ may have fewer equivalence class representatives in $\EGSet$ than the generic case.

\begin{example}
Let $p\equiv 3\pmod{4}$, $\mathbb{F}_{p^2} := \mathbb{F}_p[s]/(s^2 + 1)$, and $E_{1728}: y^2 = x^3 + x$. The order-two subgroups of $E(\Fpbar)$ are:
\[G_0:= \{(0,0),O_{E_{1728}}\},\quad G_1:= \{(s,0),O_{E_{1728}}\},\quad G_2:= \{(-s,0),O_{E_{1728}}\}.\]
The curve $E_{1728}$ has automorphism group isomorphic to $\mathbb{Z}/4\mathbb{Z}$ and is generated by $[s]:(x,y)\mapsto (-x,sy)$. The automorphism $[s]$ in particular sends $G_1$ to $G_2$ and vice versa, meaning $(E,G_1)\sim (E,G_2)$. Here, the $j$-invariant $1728$ only has two distinct equivalence classes in $|\mathcal{S}_2|$ instead of three, which is the generic case. 
\end{example}

In order to better understand the number of elements $(E,G)\in\EGSet$ with $j(E)=1728$, we write down a matrix representation for the action of $[i]$ on $E[N]\cong\ZZ/N\ZZ\times\ZZ/N\ZZ$ for any $N$. A similar procedure works for $j(E) = 0$ as well. For simplicity of enumerating the subgroups, we restrict to the case $N$ prime, but we hope that even with these simplifications the general principle will be clear to the reader.

\begin{proposition}
Let $E/\overline{\mathbb{F}}_p$ be an elliptic curve with $j(E) = 1728$, and let $N\neq p$ be prime. The action of $[i]$ on the set of order $ N$ subgroups of $E[N]$ is as follows:
\begin{enumerate}
    \item If $N$ is odd and $\left(\frac{-1}{N}\right) = -1$, then $[i]$ permutes the order $N$ subgroups in 2-cycles.
    \item If $N$ is odd and $\left(\frac{-1}{N}\right) = 1$, then $[i]$ fixes two order $N$ subgroups, and permutes the remaining $ N-1$ order $ N$ subgroups in 2-cycles. 
    \item If $N$ is even, then $[i]$ fixes one order $N$ subgroup and permutes the remaining two. 
\end{enumerate}
\end{proposition}

\begin{proof}
We begin by choosing a basis of $E[N]$ in order to write down a matrix representation of the action of $[i]$ on $E[N]$, for any integer $N>1$. Begin by noting that the action of $[i]$ on $E[N]$ cannot be the same as the action of $[m]$, the multiplication-by-$m$ map for any integer $m$. Assume $m<N$, as the action on the $N$-torsion is not changed by reducing $m$ modulo $N$. If it were, then $E[N]\subseteq \ker([m] - [i])$, which implies $N^2\mid \deg([m] - [i]) = m^2+1$. This contradicts the assumption that $m<N$. Thus there exists some $P\in E[N]$ such that $[i](P)$ is not a scalar multiple of $P$. We choose $P, Q:=[i](P)$ as our basis for $E[N]$. The matrix representation of $[i]$ with respect to this basis is then:
\[M_{[i]} = \begin{pmatrix}
0 & -1\\
1 & 0
\end{pmatrix}.\]
Enumerate the subgroups of order $N$ with respect to this basis as follows:
    \[G_0 = \langle P \rangle, \, G_1 = \langle P+Q\rangle,\,\dots\, ,G_{N-1} = \langle P + (N-1)Q\rangle\, ,G_N=\langle Q\rangle.\]
By construction, $M_{[i]}(G_0) = G_N$ and $M_{[i]}(G_N) = G_0$. 

Suppose $N$ is odd. Take $k\in\{1,2,..., N-1\}$. The group $G_k$ is generated by $\begin{pmatrix}1\\k\end{pmatrix}$, which $[i]$ maps to:
\[\begin{pmatrix}0 & -1\\ 1&0\end{pmatrix}\begin{pmatrix}1\\k\end{pmatrix} =\begin{pmatrix}-k\\1\end{pmatrix}.\]
The map $[i]$ fixes $G_k$ precisely when $\begin{pmatrix}-k\\1\end{pmatrix} = \begin{pmatrix}s\\sk\end{pmatrix}$, for some integer $s$, modulo $ N$. This gives the system of congruences:
\[s \equiv -k\pmod{N}\]
\[1\equiv sk\pmod{N}\]
which leads to the equation $-1\equiv k^2\pmod{N}$. There are either 0 or 2 solutions to this equation, depending on the value of the Legendre symbol $\left(\frac{-1}{N}\right)$. Note that $1^2 =1$ and $(-1)^2=1$, so $k$ cannot be $1$ or $ N-1$ for this to be true, so these subgroups are never fixed.

For $N = 2$, direct computation shows that $[i]$ will fix precisely one subgroup of order 2 and permute the remaining two.
\end{proof}

\subsection{Eichler Orders}\label{ssec:EOS}

The classical origins of Eichler orders can be traced to papers of Eichler himself \cite{Eichler01}, \cite{Eichler2}. The theory of Eichler orders was further developed by Pizer \cite{PizerTypeNumbs}. Eichler orders of squarefree level are called \textit{hereditary}. For relevant properties and background on Eichler and hereditary orders, see \cite{QASVOIGHT}.

Any Eichler order in a quaternion algebra is the intersection of two (not necessarily distinct) maximal orders. The level of an Eichler order in $B_{p,\infty}$ is given by its index in one of the maximal orders whose intersection defines the order (this index will be the same for either order). In \cite[Lemma 8]{KLPT}, the authors describe how an Eichler order of level $N$ is equivalent data to two maximal orders with a connecting ideal of reduced norm $N$. We let $\Nrd(I)$ denote the reduced norm of the ideal $I$.

\begin{theorem}\label{thm:O(E,G)_is_EO}
$\OO(E,G)$ is isomorphic to an Eichler order of level $|G| = N$. 
\end{theorem}

\begin{proof}
Proposition~\ref{prop:Int_of_two_MaxOrds} shows that $\OO(E,G) \cong \End(E)\cap (\frac{1}{\deg\varphi}\widehat{\varphi}\End(E/G)\varphi)$, where $E/G$ is the codomain of $\varphi$. Fix $B_{p,\infty}^E:= \End(E)\otimes_\ZZ\QQ$. By the Deuring correspondence $\End(E)$ and $\frac{1}{\deg\varphi}\widehat{\varphi}\End(E/G)\varphi$ are maximal orders in the quaternion algebra $B_{p,\infty}^E$. The intersection of two maximal orders is an Eichler order, so it remains to show that the level of $\OO(E,G)$ is $N$. 

In Proposition~\ref{prop:Int_of_two_MaxOrds}, we introduced the isogeny $\varphi:E\to E/G$ with kernel $G$. 
This isogeny corresponds to a left ideal $I$ of the maximal order $\End(E)$, where $\Nrd(I) = \deg\varphi = N$. 
See Section~\ref{ssec:bckgrnd_isogenies} for a detailed description of this association between isogenies and left ideals. The left order of $I$, which we denote $\OO_L(I)$, is $\End(E)$. Analogously, let $\OO_R(I)$ denote the right order of $I$. 
The image of $\End(E/G)$ in $B_{p,\infty}^E$ under the embedding \eqref{eq:EndoRingToMaxOrder} is the right order of $I$. 
Together with Proposition~\ref{prop:Int_of_two_MaxOrds}, this shows that $\OO(E,G)\cong \OO_L(I)\cap \OO_R(I)$. 
By \cite[Lemma 8]{KLPT}, this is an Eichler order of level $\Nrd(I) = N$.
\end{proof}

The following proposition shows that our map $\OO(\cdot,\cdot)$ to Eichler orders of level $N$ of $B_{p,\infty}$ is surjective. 

\begin{proposition}\label{prop:Eichler_order_correspond}
Every Eichler order $\OO$ of level $N$ in $B_{p,\infty}$ is isomorphic to $\OO(E,G)$ for some pair $(E,G)$ in $\EGSet$.
\end{proposition}
\begin{proof}
Every local Eichler order $\OO$ of level $N$ is the intersection of two uniquely determined maximal orders $\OO_1,\OO_2$ such that $\OO$ is of index $N$ in both $\OO_1$ and $\OO_2$ \cite[Proposition 23.4.3]{QASVOIGHT}. Eichler orders of prime level are only non-maximal at primes which divide level, so all three orders $\OO,\OO_1,\OO_2$ lift uniquely to the global setting \cite[Theorem 9.1.1]{QASVOIGHT}. By the Deuring correspondence, fix an isomorphism $\End(E_1)\cong \OO_1$ for a supersingular elliptic curve $E_1/\Fpbar$. Let $B_{p,\infty}^{E_1} = \End(E_1)\otimes_\ZZ\QQ$.

By \cite[Lemma 8]{KLPT}, there exists a unique ideal $I$ of $B_{p,\infty}^{E_1}$ which is a left $\OO_1$-ideal and a right $\OO_2$-ideal of reduced norm $N$. This ideal determines a group $G$ of order $N$ given by the scheme theoretic intersection
\[G:=\bigcap_{\alpha\in I}E_1[\alpha]\]
where $E_1[\alpha]$ is the kernel of the endomorphism $\alpha$.
By equation \eqref{eq:EndoRingToMaxOrder} of Section~\ref{ssec:bckgrnd_multipleendosembedding}, the right order of $I$ is given by $\frac{1}{\deg\varphi}\hat{\varphi}\End(E_2)\varphi$. Since $\OO_2$ is the right order of $I$, we have $\OO_2 = \frac{1}{\deg\varphi}\hat{\varphi}\End(E_2)\varphi$.

By Proposition~\ref{prop:Int_of_two_MaxOrds}, $\OO(E_1,G) = \End(E_1)\cap\frac{1}{\deg\varphi}\hat{\varphi}\End(E_2)\varphi\cong \OO_1\cap\OO_2= \OO$.
\end{proof}

The failure of injectivity of $\OO(\cdot,\cdot)$ reveals structural properties of both the supersingular elliptic curves with level-$N$ structure and the Eichler orders of level $N$. We address this completely in Section~\ref{sec:vertices_to_EOs}.

\subsection{\texorpdfstring{$\ell$}{}-isogenies on the Quaternion Side}\label{ssec:ell_isog_on_quat_side}
Fix a small prime $\ell$ coprime to $pN$. The correspondence between $\ell$-isogenies and left ideals of a maximal order $\OO\cong \End(E)$ of reduced norm $\ell$ is well-known, as we recalled in Section~\ref{ssec:bckgrnd_isogenies}. Let $E$ be a supersingular elliptic curve over $\Fpbar$ with level-$N$ structure $G$. Let $\varphi:E\to E'$ be a degree-$\ell$ isogeny. Then $G\subset E[N], \varphi(G)\subset (E')[N]$ and $|G| = |\varphi(G)|=N$. Let $I_G$ be the left ideal of the maximal order isomorphic to $\End(E)$ in $B_{p,\infty}^E$ which corresponds to $\varphi$. The degree of the isogeny is the norm of the ideal. The isogeny $\varphi:E\to E'$ is also a morphism between elements of $\EGSet$, as $\varphi:(E,G)\to (E',\varphi(G))$. In this section, we describe isogenies between elements of $\EGSet$ as quaternion objects.

\begin{proposition}\label{prop:EOidealsandMOideals}
Let $\mathcal{O}$ be an Eichler order of level-$N$ specified by the intersection $M\cap M'$ of two maximal orders $M$ and $M'$. The integral left ideals $I$ of $M$ of norm coprime to $N$ are in bijection with the left ideals of $\mathcal{O}$ of norm coprime to $N$. This bijection is realized by the map $I\mapsto I\cap \mathcal{O}$ and if $\Nrd(I)$ is coprime to $N$, then $\Nrd(I) = \Nrd(\mathcal{O}\cap I)$.
\end{proposition}
\begin{proof}
Let $\OO\subset M$ be as above. At $q\nmid N$, $\OO_q = M_q$, so the ideals of $\OO_q$ are precisely the ideals of $M_q$ intersected with $\OO_q$. At $q|N$, $M_q$ and $\OO_q$ each have a unique ideal of norm coprime to $N$, namely $M_q$ and $\OO_q$ respectively. The intersection $\OO_q = M_q\cap \OO_q$ thus realizes the bijection.

By \cite[Theorem 9.1.1]{QASVOIGHT}, the local bijection realized by intersection with $\mathcal{O}$ is in fact global.
\end{proof}

We have a way of associating the ideals $I$ of maximal orders to isogenies $E\to E'$. To extend this picture to the level structure context, we need to show that a left ideal of $\OO(E,G)$ has right order $\OO(E',G')$, for some isogeny $\varphi:E\to E'$ such that $\varphi(G) = G'$.

By Proposition~\ref{prop:EOidealsandMOideals}, every left-ideal of $\OO(E,G)$ of norm prime to $N$ is of the form $I\cap \OO(E,G)$, where $I$ is a left ideal of the maximal order $\End(E)\supseteq\OO(E,G)$. Let $\varphi_I:E\to E'$ be the isogeny determined by $I$ as in Section~\ref{ssec:bckgrnd_multipleendosembedding}, and let $G':= \varphi_I(G)$.

\begin{proposition}
Let $I\cap\OO(E,G)$ be a left ideal of $\OO(E,G)$ of norm prime to $N$. Then,
\[\OO_R(I\cap \OO(E,G)) = \frac{1}{\deg\varphi_I}\hat{\varphi}_I\OO(E',G')\varphi_I.\]
\end{proposition}
\begin{proof} We proceed by showing containment in both directions. By Proposition~\ref{prop:EOidealsandMOideals} $I\cap \OO(E,G)$ is a left ideal of the Eichler order $\OO(E,G)$. 

Take $\frac{1}{\deg\varphi_I}\hat{\varphi}_I\alpha\varphi_I\in \frac{1}{\deg\varphi_I}\hat{\varphi}_I\OO(E',G')\varphi_I$, for some $\alpha\in\OO(E',G')$. To show that 
\\$(I\cap \OO(E,G))\frac{1}{\deg\varphi_I}\hat{\varphi}_I\alpha\varphi_I \subseteq I\cap \OO(E,G)$, note that the elements of $I\cap \OO(E,G)$ are characterized by the following two properties:  
    \begin{itemize}
        \item[(i)] Every $\nu\in I\cap \OO(E,G)$ must be of the form $\beta\circ\varphi_I$, for some $\beta\in \Hom(E',E)$. This property is equivalent to being in $I$, by \cite[Lemma 42.2.7]{QASVOIGHT}.
        \item[(ii)] Every $\nu\in I\cap \OO(E,G)$ must satisfy $\nu(G)\subseteq G$. This property is equivalent to being in $\OO(E,G)$, by definition. Note that this is equivalent to requiring that $\beta(G')\subseteq G$, when we write $\nu$ in the form $\nu = \beta\circ\varphi_I$.
    \end{itemize}
    
For any $\beta\circ\varphi_I\in I\cap \OO(E,G)$ we have:
    \begin{equation*}
    \begin{split}
        \beta\circ\varphi_I\circ\left(\frac{1}{\deg\varphi_I}\hat{\varphi}_I\circ\alpha\circ\varphi_I\right) &= \beta\circ\alpha\circ\varphi_I. \\
    \end{split}
    \end{equation*}
The element $\beta\circ\alpha\circ\varphi_I$ satisfies condition (i). To check condition (ii):
    \begin{equation*}
        \beta\circ\alpha\circ\varphi_I(G)
        =\beta\circ\alpha(G')
        \subseteq\beta(G')
        \subseteq G.
    \end{equation*}

To see $\OO_R(I\cap \OO(E,G)) \subseteq \frac{1}{\deg\varphi_I}\hat{\varphi}_I\OO(E',G')\varphi_I$, recall that $I\cap \OO(E,G)$ is a left ideal of an Eichler order of level $N$, and the right order of this ideal must also be an Eichler order of level $N$ (see \cite[Lemma 17.4.11]{QASVOIGHT}). Since $\OO_R(I\cap\OO(E,G))$ contains the Eichler order $\frac{1}{\deg\varphi_I}\hat{\varphi}_I\OO(E',G')\varphi_I$ of level $N$, this containment is equality.
\end{proof}

For an explicit example illustrating the correspondence between supersingular elliptic curves with level-$N$ structure and Eichler orders, we refer the reader to Example~\ref{examp}.

\section{Failure of Injectivity of \texorpdfstring{$\OO(\cdot,\cdot)$}{O(-,-)}}\label{sec:vertices_to_EOs}
We have shown how to associate supersingular elliptic curves with level-$N$ structure to Eichler orders of $B_{p,\infty}$ of level $N$ via the map $\OO(\cdot,\cdot)$. This map is not usually bijective, and we study the properties of supersingular elliptic curves with level structure which result in the various possible fiber sizes. In this section, we restrict to the case where $N$ is squarefree, in addition to being coprime to $p$. Eichler orders of level $N$ are in fact hereditary. In Section~\ref{ssec:ActionsOnEGSet} we describe two involutions corresponding to a dualizing action and the $p$-power Frobenius action. 
These involutions help us determine the fibers of $\OO(E,G)$ in Theorem~\ref{thm:fiber_aboveO(E,G)} of Section~\ref{ssec:Fibers_Subsec}.

\subsection{Involutions}\label{ssec:ActionsOnEGSet}

In this section, we will define two group involutions on the set $\EGSet$ of equivalence classes of supersingular elliptic curves with level-$N$ structure.

We begin by defining a dualizing involution on the equivalence classes $\ec{E,G}\in\EGSet$. For $N = q_1q_2\cdots q_r$, we will have $r$ dualizing involutions. The initiated reader will recognize these dualizing involutions as Atkin-Lehner involutions, and the author would like to thank Jaap Top in particular for this helpful perspective. First, we will define and illustrate this involution when $N$ is prime:

\begin{definition}[Dualizing Involution on $\EGSet$, prime case]
\[D(E,G):= (E/G,\widehat{G}),\]
where $\varphi_G:E\to E/G$ is an isogeny with kernel $G$, and $\widehat{G}$ denotes the kernel of the dual isogeny $\hat{\varphi}_G: E/G\to E$. In particular, $\widehat{G} = \varphi_G(E[N])$.
\end{definition}

The data of $(E,G)$ is equivalent to that of an isogeny, and that isogeny has a unique dual. In this way, the data of $(E,G)$ is equivalent to the data of $(E/G,\widehat{G})$. When $N$ factors as $N = q_1q_2\cdots q_r$, the kernel $G$ factors as $G = G_1\oplus G_2\oplus\cdots \oplus G_r$. There is a $D_i$ dualizing involution for each $i = 1,2,...,r$, defined as follows:

\begin{definition}[Dualizing Involution on $\EGSet$]\label{def:dualizing_action} 
Let $N = q_1\cdots q_r$, so $G = G_1\oplus G_2\oplus\cdots \oplus G_r$. Take $i\in\{1,...,r\}$ and without loss of generality let $i = 1$. The isogeny with kernel $G$ can be factored as $\varphi_G = \varphi_r\circ\varphi_{r-1}\circ\cdots\circ\varphi_2\circ\varphi_1$ with $\ker\varphi_1 = G_1$, $\ker\varphi_2 = \varphi_1(G_2)$, ..., $\ker\varphi_r = \varphi_{r-1}(\varphi_{r-2}(\cdots(\varphi_1(G_r))))$. Define $\widehat{G}_1 := \varphi_1(E[q_1])$, which is the kernel of $\widehat{\varphi}_1$. Define $\widehat{G}:=\widehat{G}_1\oplus\varphi_1(G_2)\oplus\varphi_1(G_3)\oplus \cdots \oplus \varphi_1(G_r)$, which is a subgroup of $\varphi_1(E)[N]$ of order $N$. Finally, we define:
\[D_1(E,G) := (\varphi_1(E),\widehat{G}).\]
\end{definition}

For $N = q_1\cdots q_r$, there are $r$ distinct dualizing involutions. First, we show that they are compatible with the equivalence relation on $\EGSet$ and then we show that they are commutative.

\begin{lemma}
If $(E,G)\sim (E',G')$ as in Definition~\ref{def:SSECsWithLS}, then $D_i(E,G)\sim D_i(E',G')$.
\end{lemma}
\begin{proof}
Without loss of generality, we take $i = 1$. We can factor $N = q_1\cdots q_r$, $G = G_1\oplus\cdots\oplus G_r$, and $G' = G_1'\oplus \cdots\oplus G_r'$ with $|G_i| = |G_i'| = q_i$. 

Suppose $\eta: E \to E'$ is the isomorphism sending $G$ to $G'$. This gives the following commutative diagram:
\begin{center}
\begin{tikzcd}
O_E \arrow[r] & G \arrow[r,hook] \arrow[d,"\eta"] & E \arrow[r, "\varphi_{G_1}"] \arrow[d,"\eta"] & E/G_1 \arrow[r] \arrow[d,dashed,"\omega"] & O_{E/G_1}\arrow[d,"\eta"]\\
O_{E'} \arrow[r] & G_1' \arrow[r,hook] & E' \arrow[r,"\varphi_{G_1'}"] & E'/G_1' \arrow[r] & O_{E'/G_1'}
\end{tikzcd}
\end{center}
Since $\omega\circ\varphi_{G_1} = \varphi_{G_1'}\circ\eta$ and $\deg\varphi_{G_1} = \deg\varphi_{G'_1}$ and $\eta$ is an isomorphism, the map $\omega: E/G_1\to E'/G_1'$ is an isomorphism. Moreover, $\widehat{\varphi_{G_1}} = \varphi_{\widehat{G}_1}= \widehat{\eta}\circ\varphi_{\widehat{G}_1'}\circ\omega$. Comparing the kernels on the left and right, we see $\omega(\widehat{G}_1)\subseteq\ker\varphi_{\widehat{G}_1'}$, and so $\omega(\widehat{G}_1) = \widehat{G}_1'$. 

We also have $\eta(G_i) = G_i'$ for all $i\neq 1$, so we can construct:
\[\widehat{G} = \widehat{G}_1\oplus\varphi_{G_1}(G_2)\oplus \varphi_{G_1}(G_3)\oplus\cdots\oplus\varphi_{G_1}(G_r)\]
and \[\widehat{G}' = \widehat{G}'_1\oplus\varphi_{G_1'}(G_2')\oplus \varphi_{G_1'}(G_3')\oplus\cdots\oplus\varphi_{G_1'}(G_r')\]
where $\widehat{G}\cong\widehat{G}'$ via $\omega$, and so we have $D_1(E,G)\sim D_1(E',G')$.
\end{proof}

\begin{lemma}
Let $(E,G)\in\EGSet$ for $N = q_1\cdots q_r$ and factor $G = G_1\oplus\cdots\oplus G_r$ with $|G_i| = q_i$.
Let $D_i,D_j$ be any two dualizing involutions with $i\neq j$. Then, $D_j(D_i(E,G))\sim D_i(D_j(E,G))$.
\end{lemma}

\begin{proof}
Without loss of generality, we may reorder the factors of $N = q_1\cdots q_r$ in such a way that $i = 1$ and $j = 2$. 
Let $\varphi_1$ denote the isogeny from $E$ with kernel $G_1$ and let $\varphi_2$ denote the isogeny from $E$ with kernel $G_2$. Let $\varphi_2'$ denote the isogeny from $E/G_1$ with kernel $\varphi_1(G_2)$ and let $\varphi_1'$ denote the isogeny from $E/G_2$ with kernel $\varphi_2(G_1)$. 

By definition, we have:
\[D_2(D_1(E,G)) = ((E/G_1)/\varphi_1(G_2), \varphi_2'(\widehat{G}_1)\oplus \ker\widehat{\varphi_2'}\oplus \varphi_2'(\varphi_1(G_3))\oplus \cdots \oplus \varphi_2'(\varphi_1(G_r)),\]
\[D_1(D_2(E,G)) = ((E/G_2)/\varphi_2(G_1), \ker\widehat{\varphi_1'}\oplus \varphi_1'(\widehat{G}_2)\oplus \varphi_1'(\varphi_2(G_3))\oplus \cdots \oplus \varphi_1'(\varphi_2(G_r)) ).\]

There exist isomorphisms 
\[\eta_1:(E/G_1)/\varphi_1(G_2) \to E/(G_1\oplus G_2),\] 
\[\eta_2:(E/G_2)/\varphi_2(G_1)\to E/(G_1\oplus G_2),\]
since $G_1\oplus G_2$ is the kernel of the composition of the two quotients, in both cases.
So $\eta:=\eta_1^{-1}\circ\eta_2$ is an isomorphism $(E/G_2)/\varphi_2(G_1)\cong (E/G_1)/\varphi_1(G_2)$. This is summarized in the following commutative diagram where we see $\eta\circ\varphi_1'\circ\varphi_2 = \varphi_2'\circ\varphi_1$:
\begin{center}
\begin{tikzcd}
E \arrow[r,"\varphi_1"] \arrow[d,equal]& E/G_1\arrow[r,"\varphi_2'"] & E/(G_1\oplus\varphi_1(G_2))\\
E \arrow[r,"\varphi_2"]& E/G_2\arrow[r,"\varphi_1'"] & E/(\varphi_2(G_1)\oplus G_2) \arrow[u, "\eta"]
\end{tikzcd}
\end{center}
We will show that the isomorphism $\eta$ satisfies the following three properties:
\begin{enumerate}
    \item $\eta(\ker\widehat{\varphi_1'}) = \varphi_2'(\widehat{G}_1)$,
    \item $\eta(\varphi_1'(\widehat{G}_2)) = \ker\widehat{\varphi_2'}$,
    \item $\eta(\varphi_1'(\varphi_2(G_i))) = \varphi_2'(\varphi_1(G_i))$ for all $i = 3, 4,\dots,r$.
\end{enumerate}

For the first property, plug the $q_1$-torsion of $E$ into both sides of the equation $\eta\circ\varphi_1'\circ\varphi_2 = \varphi_2'\circ\varphi_1$:
\begin{equation*}
\begin{split}
    \eta\circ\varphi_1'\circ\varphi_2(E[q_1]) &= \varphi_2'\circ\varphi_1(E[q_1])\\
    \eta\circ\varphi_1'((E/G_2)[q_1]) &= \varphi_2'(\widehat{G}_1)\\
    \eta(\ker\widehat{\varphi_1'}) &= \varphi_2'(\widehat{G}_1)
\end{split}
\end{equation*}

The second property follows symmetrically after applying $\eta^{-1}$ to both sides.

The third property is immediate by plugging $G_i$ into both sides of the equation $\eta\circ\varphi_1'\circ\varphi_2 = \varphi_2'\circ\varphi_1$, for $i\in\{3,\dots,r\}$.

We have established $D_1(D_2(E,G))\sim D_2(D_1(E,G))$ via the isomorphism $\eta(D_1(D_2(E,G))) = D_2(D_1(E,G))$.
\end{proof}

The $p$-power Frobenius map $\pi_p^E:E\to E^p$ defines another involution map $F_p$ on supersingular elliptic curves with level-$N$ structure in the following manner:

\begin{definition}[Frobenius Involution on $\EGSet$]
\[F_p\ec{E,G} := \ec{E^p,G^p},\]
where $E^p$ is the codomain of $\pi_p^E:E\to E^p$ and $G^p = \pi_p^E(G)$. 
\end{definition}
Lemma~\ref{lem:FrobInvIsWellDef} shows that this Frobenius involution gives an involution $F_p$ on the set of equivalence classes $\EGSet$

\begin{lemma}\label{lem:FrobInvIsWellDef}
If $(E,G)\sim (E',G')$ as in Definition~\ref{def:SSECsWithLS}, then $F_p(E,G)\sim F_p(E',G')$.
\end{lemma}
\begin{proof}
Suppose $\alpha:E'\to E$ is the isomorphism such that $\alpha(G') = G$. Let $\pi_p^{E'}:E'\to (E')^p$ denote the $p$-power Frobenius map of $E'$. Since $\alpha$ is separable and $\ker\alpha = O_{E'}\subseteq \ker\pi_p^{E'}$, there exists a unique isogeny $\lambda:E\to (E')^p$ such that $\pi_p^{E'} = \lambda\circ\alpha$ {\cite[Corollary III.4.11]{AEC}}. See Figure~\ref{fig:FpWellDefined01}.

\begin{figure}
    \centering
    \hfill
\begin{subfigure}{.45\textwidth}
\centering
\begin{tikzcd}
E' \arrow[rr,"\alpha"]\arrow[dd,swap, "\pi_p^{E'}"] & & E \arrow[ddll,dashed,"\lambda"]\\
& & \\
(E')^p & & 
\end{tikzcd}
\caption{Given an isomorphism $\alpha:E'\to E$ and the $p$-power Frobenius $\pi_p:E'\to (E')^p$, there exists a unique isogeny $\lambda:E\to (E')^p$ such that $\lambda\circ\alpha = \pi_p^{E'}$.}
\label{fig:FpWellDefined01}
\end{subfigure}
\hfill
\begin{subfigure}{.45\textwidth}
\centering 
\begin{tikzcd}
E' \arrow[rr,"\alpha"]\arrow[dd,swap, "\pi_p^{E'}"] & & E \arrow[dl,dashed,"\pi_p^E"]\\
& E^p\arrow[dl,"\beta"]& \\
(E')^p & & 
\end{tikzcd}
\caption{By inseparable degrees of the right and left sides of $\lambda\circ\alpha = \pi_p^{E'}$, we decompose $\lambda = \beta\circ\pi_p^E$ where $\beta:E^p\to (E')^p$ is an isomorphism.}
\label{fig:FpWellDefined02}
\end{subfigure}
\hfill
    \caption{Diagram to support Lemma~\ref{lem:FrobInvIsWellDef} in proving $F_p$ is well-defined on equivalence classes of $\EGSet$.}
    \label{fig:my_label}
\end{figure}

By comparing the total and inseparable degrees of the left and right sides of $\pi_p^{E'} = \lambda\circ\alpha$, we see that $\lambda$ decomposes as $\lambda = \beta\circ\pi_p^E$, where $\beta$ is separable and of degree 1. See Figure~\ref{fig:FpWellDefined02}.

Thus, $\beta:E^p\to (E')^p$ is an isomorphism. Moreover, $\beta\circ\pi_p^E\circ\alpha = \pi_p^{E'}$, and plugging in $G'$ we have:
\begin{equation*}
\begin{split}
    \pi_p^{E'}(G') & = \beta\circ\pi_p^E\circ\alpha(G')\\
    (G')^p &= \beta\circ\pi_p^E(G)\\
    (G')^p &= \beta(G^p)
\end{split}
\end{equation*}
so $\beta$ realizes $(E^p,G^p)\sim ((E')^p,(G')^p)$. 
\end{proof}

\begin{lemma}\label{lem:DandFpCommute}
For all $i = 1,...,r$,
\[D_i(F_p(E,G)) \sim F_p(D_i\ec{E,G}).\]
\end{lemma}
\begin{proof}
Without loss of generality, we assume $D_i = D_1$ corresponding to the factorizations $N = q_1\cdots q_r$ with $q_i$ distinct primes and $G = G_1\oplus \cdots \oplus G_r$ with $|G_i| = q_i$.

With the notation established, we wish to show that
\[((E/G_1)^p,(\widehat{G}_1)^p\oplus \varphi_1(G_2)^p\oplus\cdots \oplus \varphi_1(G_r)^p)\sim (E^p/G_1^p, \widehat{G^p_1}\oplus\varphi_{G_1^p}(G_2^p)\oplus\cdots \varphi_{G_1^p}(G_r^p)).\]

Since $\varphi_1:E\to E/G_1$ is separable and $\varphi_{G_1^p}\circ\pi_p^E(G_1) = O_{E^p/G_1^p}$, there exists a unique isogeny $\lambda:E/G_1\to E^p/G_1^p$ such that $\lambda\circ\varphi_1 = \varphi_{G_1^p}\circ\pi_p^E$. See Figure~\ref{fig:DandFpCommute01}. 
By comparing the total and inseparable degrees of the left and right sides of $\lambda\circ\varphi_1 = \varphi_{G_1^p}\circ\pi_p^E$, we see that $\lambda$ decomposes as $\alpha\circ\pi_p^{E/G_1}$, where $\pi_p^{E/G_1}$ is the $p$-power Frobenius map on $E/G_1$ and $\alpha$ is an isomorphism. See Figure~\ref{fig:DandFpCommute02}.

Finally, to establish that $D_1(F_p(E,G))\sim F_p(D_1(E,G))$, we must show \[\alpha((\widehat{G}_1)^p\oplus \varphi_1(G_2)^p\oplus\cdots \oplus \varphi_1(G_r)^p) = \widehat{G^p_1}\oplus\varphi_{G_1^p}(G_2^p)\oplus\cdots \varphi_{G_1^p}(G_r^p),\]
which can be accomplished by showing the following two properties:
\begin{enumerate}
    \item $\alpha((\widehat{G}_1)^p) = \widehat{G_1^p}$ and
    \item $\alpha(\varphi_1(G_i)^p) = \varphi_{G_1^p}(G_i^p)$ for all $i = 2, \dots, r$.
\end{enumerate}

For the first property, plug the $q_1$-torsion of $E$ into both sides of the equation $\alpha\circ\pi_p^{E/G_1}\circ\varphi_1 = \varphi_{G_1^p}\circ\pi_p^E$:
\begin{equation*}
\begin{split}
\alpha\circ\pi_p^{E/G_1}\circ\varphi_1 (E[q_1])&= \varphi_{G_1^p}\circ\pi_p^E(E[q_1]) \\
\alpha\circ\pi_p^{E/G_1}(\widehat{G}_1) &= \varphi_{G_1^p}(E^p[q_1])\\
\alpha((\widehat{G}_1)^p) &= \widehat{G_1^p}.
\end{split}
\end{equation*}

For the second property, we proceed in the same way by plugging $G_i$ into both sides of the equation $\alpha\circ\pi_p^{E/G_1}\circ\varphi_1 = \varphi_{G_1^p}\circ\pi_p^E$:
\begin{equation*}
\begin{split}
    \alpha\circ\pi_p^{E/G_1}\circ\varphi_1(G_i) &= \varphi_{G_1^p}\circ\pi_p^E(G_i)\\
    \alpha(\varphi_1(G_i)^p) &= \varphi_{G_1^p}(G_i^p).
\end{split}
\end{equation*}

\begin{figure}
    \centering
    \hfill 
    \begin{subfigure}{.48\textwidth}
    \begin{tikzcd}
    E \arrow[r,"\pi_p^E"]\arrow[dd, "\varphi_1",swap] & E^p \arrow[r,"\varphi_{G_1^p}"] & E^p/G_1^p\\
     & & \\
     E/G_1 \arrow[uurr,dashed,swap,"\lambda"]
    \end{tikzcd}
    \caption{Given the separable isogeny $\varphi_1:E\to E/G_1$ and the isogeny $\varphi_{G_1^p}\circ\pi_p^E$, there exists a unique isogeny $\lambda:E/G_1\to E^p/G_1^p$ such that $\lambda\circ\varphi_1 = \varphi_{G_1^p}\circ\pi_p^E$.}
    \label{fig:DandFpCommute01}
    \end{subfigure}
    \hfill
    \begin{subfigure}{.48\textwidth}
        \begin{tikzcd}
    E \arrow[r,"\pi_p^E"]\arrow[dd,swap,"\varphi_1"] & E^p\arrow[r, "\varphi_{G_1^p}"] & E^p/G_1^p\\
     &(E/G_1)^p \arrow[ur,swap,"\alpha"]& \\
     E/G_1 \arrow[ur,swap,"\pi_p^{E/G_1}"]
    \end{tikzcd}
    \caption{By inseparable degrees of the right and left sides of $\lambda\circ\varphi_1 = \varphi_{G_1^p}\circ\pi_p^E$, we decompose $\lambda = \alpha\circ\pi_p^{E/G_1}$, where $\alpha:(E/G_1)^p \to E^p/G_1^p$ is an isomorphism.}
    \label{fig:DandFpCommute02}
    \end{subfigure}
    \hfill 
    \caption{Diagram to support Lemma~\ref{lem:DandFpCommute} in proving that $F_p$ and $D_1$ commute on equivalence classes of $\EGSet$.}
    \label{fig:my_label2}
\end{figure}
\end{proof}

\subsection{Fibers of \texorpdfstring{$\OO(\cdot,\cdot)$}{O(-,-)}}\label{ssec:Fibers_Subsec}

The involutions $D_i$ and $F_p$ descend to well-defined involutions on isomorphism classes of Eichler orders via the map $\OO(\cdot,\cdot)$ as defined in Definition~\ref{def:OO(E,G)}. In this section, we show that these descended involutions are all trivial on isomorphism classes of Eichler orders. As in the previous section, we factor $N$ into distinct prime factors $N = q_1\cdots q_r$, we decompose $G = G_1\oplus\cdots\oplus G_r$ with $|G_i| = q_i$, and we let $D_i$ denote the dualizing involutions for $i\in\{1,\dots,r\}$.

\begin{lemma}\label{lem:D_act_triv}
For any $\ec{E,G}\in \EGSet$ and any $i\in \{1,\dots,r\}$, $\OO( D_i\ec{E,G}) \cong\OO\ec{E,G}$. Moreover, the isomorphism is given by an isomorphism of the quaternion algebras $B_{p,\infty}^{E/G_i}\to B_{p,\infty}^E$.
\end{lemma}
\begin{proof}
Without loss of generality, reorder the factors of $N$, $G$ such that $D_i = D_1$. By definition of the dualizing involution, we have:
\[D_1(E,G) = (E/G_1,\widehat{G}_1\oplus\varphi_1(G_2)\oplus\cdots\oplus\varphi_1(G_r)),\]
where $\varphi_1:E\to E/G_1$ is the isogeny with $\ker\varphi_1 = G_1$ and where $\widehat{G}_1 = \ker\widehat{\varphi}_1$. To ease notation, 
\[D_1(G):=\widehat{G}_1\oplus\varphi_1(G_2)\oplus\cdots\oplus\varphi_1(G_r).\]

By Proposition~\ref{prop:Int_of_two_MaxOrds}, 
\[\OO(E,G) = \End(E)\cap\left(\frac{1}{\deg\varphi_G}\widehat{\varphi}_G\End(E/G)\varphi_G\right)\subseteq B_{p,\infty}^E,\] 
where $\varphi_G:E\to E/G$ is an isogeny with $\ker\varphi_G = G$. Likewise,
\[\OO(D_1(E,G)) = \End(E/G_1)\cap \left(\frac{1}{\deg\varphi_{D_1(G)}}\widehat{\varphi}_{D_1(G)}\End((E/G_1)/D_1(G))\varphi_{D_1(G)}\right)\subseteq B_{p,\infty}^{E/G_1},\]
where $\varphi_{D_1(G)}:E/G_1\to (E/G_1)/D_1(G)$ is the isogeny with kernel $D_1(G)$.

The map from $B_{p,\infty}^{E/G_1}\to B_{p,\infty}^E$ is given by conjugation by $\frac{1}{q_1}\widehat{\varphi}_1(-)\varphi_1$. Mapping $\OO(D_1(E,G))$ into $B_{p,\infty}^E$ by this map,  we wish to show 
\[\frac{1}{q_1}\widehat{\varphi}_1\OO(D_1(E,G)))\varphi_1 = \OO(E,G).\]
We will achieve this by showing containment in one direction. As these are Eichler orders of the same level, this suffices to show equality. 

Take $\alpha\in\OO(D_1(E,G))$. By definition, $\alpha$ satisfies the following two properties:
\begin{enumerate}
    \item $\alpha(\widehat{G}_1)\subseteq \widehat{G}_1$ and
    \item $\alpha(\varphi_1(G_i))\subseteq \varphi_1(G_i)$ for $i = 2, \dots, r$.
\end{enumerate}
We need to show that $\frac{1}{q_1}\widehat{\varphi}_1\circ\alpha\circ\varphi_1$ satisfies the following two properties:
\begin{enumerate}
    \item $\frac{1}{q_1}\widehat{\varphi}_1\circ\alpha\circ\varphi_1(G_1)\subseteq G_1$ and 
    \item $\frac{1}{q_1}\widehat{\varphi}_1\circ\alpha\circ\varphi_1(G_i)\subseteq G_i$ for $i = 2,\dots,r$.
\end{enumerate}

For the first property, we plug $G_1$ into the equation and use the fact that $G_1 = \widehat{\varphi}_1((E/G_1)[q_1])$ and $\alpha((E/G_1)[q_1])\subseteq (E/G_1)[q_1]$:
\[\frac{1}{q_1}\widehat{\varphi}_1\circ\alpha\circ\varphi_1(G_1) = \frac{1}{q_1}\widehat{\varphi}_1\circ\alpha\circ\varphi_1(\widehat{\varphi}_1((E/G_1)[q_1])) = \widehat{\varphi}_1\circ\alpha((E/G_1)[q_1])\subseteq\widehat{\varphi}_1((E/G_1)[q_1]) = G_1.\]

For the second property, take $G_i$ with $i\in\{2,\dots,r\}$ and plug into the equation and use the fact that $\alpha(\varphi_1(G_i))\subseteq\varphi(G_i)$:
\[\frac{1}{q_1}\widehat{\varphi}_1\circ\alpha\circ\varphi_1(G_i) \subseteq \frac{1}{q_1}\widehat{\varphi}_1(\varphi_1(G_i)) = G_i.\]

This establishes $\frac{1}{q_1}\widehat{\varphi}_1\OO(D_1(E,G)))\varphi_1 = \OO(E,G)$, and thus $\OO(D_1(E,G))) \cong \OO(E,G)$, where the isomorphism is the map $B_{p,\infty}^{E/G_1}\to B_{p,\infty}^E$ given by conjugation by $\varphi_1$.
\end{proof}

Denote an arbitrary composition of $D_i$ involutions as $D_J(E,G)=(\varphi_J(E),G_J)$, where $J\subseteq\{1,\dots,r\}$ (including the possibility $J = \varnothing$), to ease notation.
We see that there are $2^r$ possible compositions of dualizing involutions: $(E,G), D_1(E,G), D_2(E,G),...,D_2\circ D_1(E,G), ...$, etc. 
For every arbitrary composition $D_J(E,G)$, we have $\OO(E,G) = \OO(D_J(E,G))$. 

\begin{lemma}\label{lem:Fp_act_triv}
For any $\ec{E,G}\in \EGSet$, $\OO( F_p\ec{E,G}) \cong \OO\ec{E,G}$. Moreover, the isomorphism is given by an isomorphism of the quaternion algebras $B_{p,\infty}^{E^p}\to B_{p,\infty}^E$.

\end{lemma}
\begin{proof}
By Proposition~\ref{prop:Int_of_two_MaxOrds}, 
\[\OO(E,G) = \End(E)\cap\left(\frac{1}{\deg\varphi_G}\widehat{\varphi}_G\End(E/G)\varphi_G\right) \subseteq B_{p,\infty}^E,\] 
where $\varphi_G:E\to E/G$ is an isogeny with $\ker\varphi_G = G$. Likewise,
\[\OO(E^p,G^p) = \End(E^p)\cap\left(\frac{1}{\deg\varphi_{G^p}}\widehat{\varphi}_{G^p}\End(E^p/G^p)\varphi_{G^p}\right)\subseteq B_{p,\infty}^{E^p},\] 
where $\varphi_{G^p}:E^p\to E^p/G^p$ is an isogeny with $\ker\varphi_{G^p} = G^p = \pi_p(G)$, where $\pi_p$ is the $p$-power Frobenius isogeny from $E$ to $E^p$. By Lemma~\ref{lem:DandFpCommute}, we have a relationship between $\varphi_G$ and $\varphi_{G^p}$:
\[\varphi_{G^p}\circ\pi_p^E = \lambda\circ\varphi_G,\]
where $\lambda = \alpha\circ\pi_p^{E/G}$ for an isomorphism $\alpha:(E/G)^p\to E^p/G^p$.

Next, we map $\OO(E^p,G^p)$ into $B_{p,\infty}^{E}$. For what follows, write $\pi_p:= \pi_p^E$. Since $\pi_p:E\to E^p$, we have $\hat{\pi_p}:E^p\to E$. This map gives an isomorphism $\frac{1}{p}\hat{\pi_p}\End(E^p)\pi_p= \End(E)$. Conjugating $\OO(E^p,G^p)$ by this map, we obtain the image of $\OO(E^p,G^p)$ in $B_{p,\infty}^E$:
\begin{equation*}
\begin{split}
    \frac{1}{p}\widehat{\pi_p}\OO(E^p,G^p)\pi_p &= \frac{1}{p}\hat{\pi_p}\left(\End(E^p)\cap(\frac{1}{\deg\varphi_{G^p}}\widehat{\varphi}_{G^p}\End(E^p/G^p)\varphi_{G^p})\right)\pi_p\subseteq B_{p,\infty}^E\\
    & = \frac{1}{p}\hat{\pi_p}\End(E^p)\pi_p\cap\left(\frac{1}{p\deg\varphi_{G^p}}\widehat{\varphi_{G^p}\pi_p}\End(E^p/G^p)\varphi_{G^p}\pi_p\right)\\
    &= \End(E)\cap \left(\frac{1}{p\deg\varphi_G}\widehat{\lambda\varphi_G}\End(E^p/G^p)\lambda\varphi_G\right)
\end{split}
\end{equation*}
Recall that $\lambda:E/G\to E^p/G^p$ factors as $\lambda = \alpha\circ\pi_p^{E/G}$, where $\alpha:(E/G)^p\to E^p/G^p$ is an isomorphism. Substituting this in above gives:
\begin{equation*}
\begin{split}
    \frac{1}{p}\widehat{\pi_p}\OO(E^p,G^p)\pi_p &= \End(E) \cap \left(\frac{1}{p\deg\varphi_G}\widehat{\varphi_{G}}\widehat{\pi_p^{E/G}}\widehat{\alpha}\End(E^p/G^p)\alpha\pi_p^{E/G}\varphi_{G}\right)\\
    &= \End(E) \cap \left(\frac{1}{p\deg\varphi_G}\widehat{\pi_p^{E/G}\varphi_{G}}\End((E/G)^p)\pi_p^{E/G}\varphi_{G}\right)\\
    &= \End(E) \cap \left(\frac{1}{\deg\varphi_G}\widehat{\varphi_{G}}\End(E/G)\varphi_{G}\right)\\
    &=\OO(E,G),
\end{split}
\end{equation*}
where the equality between lines (2) and (3) follows from $\frac{1}{p}\widehat{\pi_p^{E/G}}\End((E/G)^p)\pi_p^{E/G} = \End(E/G)$.
We have recovered $\OO(E^p,G^p)\cong\OO(E,G)$, where the isomorphism is the map $B_{p,\infty}^{E^p}\to B_{p,\infty}^E$ given by conjugation by $\pi_p$. 
\end{proof}

%\subsection{Explicit Fibers of $\mathcal{O}(\cdot,\cdot)$}\label{ssec:p1Mod12}

The fiber along $\OO(\cdot,\cdot)$ above $\OO(E,G)$ contains $D_J(E,G),D_J(E^p,G^p)\in \EGSet$. 
We will see that these are the only possible elements of the fiber along $\OO(\cdot,\cdot)$ in Theorem~\ref{thm:fiber_aboveO(E,G)}. 
However, this does not mean that the fiber along $\OO(\cdot,\cdot)$ is always of the same size: it can happen that two or more of the equivalence classes listed above coincide.

\begin{example}[$p = 61$, $N = 2$]\label{ex:p61N2Sets}
Let $\mathbb{F}_{61^2} = \mathbb{F}_{61}[s]/(s^2 + 60s + 2)$. Table~\ref{tbl:p61N2JInvs} lists supersingular $j$-invariants and Weierstrass equations.

\begin{table}
\begin{tabular}{|l|l|}
\hline
$j(E)$     & Weierstrass Equation               \\ \hline
$9$        & $E_9:y^2 = x^3 + 53x + 18$             \\ \hline
$41$       & $E_{41}:y^2 = x^3 + 6x + 34$              \\ \hline
$50$       & $E_{50}: y^2 = x^3 + 14x + 36$     \\ \hline
$20s + 32$ & $E_{20s+32}: y^2 = x^3 + (30s+47)x + (48s+49)$ \\ \hline
$41s+52$   & $E_{41s+52}: y^2 = x^3 + (31s+16)x + (13s+36)$ \\ \hline
\end{tabular}
\caption{$j$-invariants and Weierstrass equations for the computations of Example~\ref{ex:p61N2Sets}.}
\label{tbl:p61N2JInvs}
\end{table}

Table~\ref{tbl:setsforp61N2} sorts the pairs $(E,G)$ into sets of the form:
\[\{(E,G), (E/G,\hat{G}), (E^p,G^p),((E/G)^p,\hat{G}^p)\}.\] 
The last column indicates the size of the set $\{(E,G), (E/G,\hat{G}), (E^p,G^p),((E/G)^p,\hat{G}^p)\}$, i.e., the size of the fiber above the corresponding image under $\OO(\cdot,\cdot)$.

\begin{table}
\begin{center}
\scalebox{.8}{\begin{tabular}{|l|l|l|l|l|}
\hline
$(E,G)$                                 & $(E/G, \hat{G})$                              & $(E^p,G^p)$                             & $((E/G)^p,\hat{G}^p)$                       & $\mid$Set$\mid$ \\ \hline
$(E_{50},\langle(59,0)\rangle)$         & $(E_{41},\langle (4,0)\rangle)$         & $(E_{50},\langle(59,0)\rangle)$         & $(E_{41},\langle (4,0)\rangle)$         & 2        \\ \hline
$(E_{50},\langle(60s+32,0)\rangle)$     & $(E_{20s+32},\langle(2s+58,0)\rangle)$  & $(E_{50},\langle(s+31,0)\rangle)$       & $(E_{41s+52},\langle(59s+60,0)\rangle)$ & 4        \\ \hline
$(E_{41},\langle (43s + 7,0) \rangle)$  & $(E_{41},\langle (18s + 50,0) \rangle)$ & $(E_{41},\langle (18s + 50,0) \rangle)$ & $(E_{41},\langle (43s + 7,0) \rangle)$  & 2        \\ \hline
$(E_{20s+32},\langle(40s+6,0)\rangle)$  & $(E_{41s+52},\langle(21s+46,0)\rangle)$ & $(E_{41s+52},\langle(21s+46,0)\rangle)$ & $(E_{20s+32},\langle(40s+6,0)\rangle)$  & 2        \\ \hline
$(E_{20s+32},\langle(19s+58,0)\rangle)$ & $(E_9,\langle(50s+2,0)\rangle)$         & $(E_{41s+52},\langle(42s+16,0)\rangle)$ & $(E_9,\langle(11s+52,0)\rangle)$        & 4        \\ \hline
$(E_9,\langle(7,0)\rangle)$             & $(E_9,\langle(7,0)\rangle)$             & $(E_9,\langle(7,0)\rangle)$             & $(E_9,\langle(7,0)\rangle)$             & 1        \\ \hline
\end{tabular}}
\caption{Table of the sets $\{(E,G),(E/G,\hat{G}),(E^p, G^p),((E/G)^p,\hat{G}^p)\}\subset \EGSet$ for $p = 61$, $N=2$}\label{tbl:setsforp61N2}
\end{center}
\end{table}

\end{example}

\begin{lemma}\label{lem:sameendring}
If $(E,G),(F,H)\in \EGSet$ and $\OO(E,G)\cong \OO(F,H)$, then $(F,H) = D_J(E,G)$ or $D_J(E^p,G^p)$ for some $J\subseteq\{1,...,r\}$.
\end{lemma}
\begin{proof}
    By definition,
    \[\OO(E,G) = \End(E)\cap \left(\frac{1}{\deg\varphi_G}\widehat{\varphi}_G\End(E/G)\varphi_G \right)\subseteq B_{p,\infty}^E\]
    \[\OO(F,H) = \End(F)\cap \left(\frac{1}{\deg\varphi_H}\widehat{\varphi}_H\End(F/H)\varphi_H \right)\subseteq B_{p,\infty}^F.\]
    Fix an isomorphism $B_{p,\infty}^F\cong B_{p,\infty}^E$ and work solely in $B_{p,\infty}^E$. Our assumption gives an isomorphism 
    \[\End(E)\cap \left(\frac{1}{\deg\varphi_G}\widehat{\varphi}_G\End(E/G)\varphi_G \right) \cong \End(F)\cap \left(\frac{1}{\deg\varphi_H}\widehat{\varphi}_H\End(F/H)\varphi_H \right).\]
    If we localize at any prime $q\nmid N$, this isomorphism is an equality of maximal orders. At any prime $q\mid N$, $\OO(E,G)_q$ is a local Eichler order, which is the intersection of two uniquely determined maximal orders. This gives one of two possibilities:
    \[\End(F)_q = \End(E)_q\,\text{ or }\End(F)_q = \left(\frac{1}{\deg\varphi_G}\widehat{\varphi}_G\End(E/G)\varphi_G\right)_q,\]
    where `$\cdot_q$' denotes `$\cdot\otimes_\ZZ \ZZ_q$'. Let $J = \{i\in\{1,...,r\}:\End(F)_{q_i}\neq \End(E)_{q_i}\}$ and define $(E_J,G_J):=D_J(E,G)$. Let $\varphi_J:E\to E_J$ denote the isogeny in the definition of $D_J(E,G)$. Then, $\End(F)_q=\left(\frac{1}{\deg\varphi_J}\widehat{\varphi}_J\End(\varphi_J(E))\varphi_J\right)_q$ locally at every prime $q$, so this equality holds globally as well. By the Deuring correspondence, either $F\cong \varphi_J(E)$ or $F\cong\pi_p\circ\varphi_J(E)$.

    The proof of Lemma~\ref{lem:D_act_triv} gives
    \[\OO(E,G) = \frac{1}{\deg\varphi_J}\widehat{\varphi}_J\OO(E_J,G_J)\varphi_J.\]
    Applying this above:
    \begin{equation*}
    \begin{split}
        \End(F)\cap\left(\frac{1}{\deg\varphi_H}\widehat{\varphi}_H\End(F/H)\varphi_H\right) =& \OO(E,G)\\
        =&\frac{1}{\deg\varphi_J}\widehat{\varphi}_J\OO(E_J,G_J)\varphi_J\\
        =&\left(\frac{1}{\deg\varphi_J}\widehat{\varphi}_J\End(E_J)\varphi_J\right)\cap\\
        &\left(\frac{1}{\deg\varphi_J}\widehat{\varphi}_J\left(\frac{1}{\deg\varphi_{G_J}}\widehat{\varphi}_{G_J}\End(E_J/G_J)\varphi_{G_J}\right)\varphi_J\right),
    \end{split}
    \end{equation*}
    where $\varphi_{G_J}:E_J\to E_J/G_J$ denotes the isogeny with kernel $G_J$. Since $\End(F) = \frac{1}{\deg\varphi_J}\widehat{\varphi}_J\End(E_J)\varphi_J$, the left ideal linking $\End(F)$ to $\End(F/H)$ is the same left ideal linking $\End(E_J)$ to $\End(E_J/G_J)$, up to conjugation by $\varphi_J$. The connecting ideal determines the kernel of a separable isogeny, giving either $(F,H) = (E_J,G_J)$ or $(F,H) = F_p(E_J,G_J)$.
\end{proof}

By \cite[Section 23.3.19]{QASVOIGHT}, above each of the primes dividing the reduced discriminant $pN = pq_1\cdots q_r$ of $\OO(E,G)$, there is a unique two-sided maximal ideal of order two in the two-sided ideal class group of $\OO(E,G)$, and the collection of these ideals generates the two-sided ideal class group of $\OO(E,G)$. Let $\mathfrak{p}$ denote the two-sided ideal of $\OO(E,G)$ above $p$ and let $\mathfrak{q}_i$ denote the two-sided ideal of $\OO(E,G)$ above $q_i$.

\begin{lemma}\label{lem:twosidedideals}
    Let $J\subseteq\{1,...,r\}$. Then, $(E,G)\sim F_p^e(D_J(E,G))$ if and only if $\mathfrak{p}^e\prod_{i\in J}\mathfrak{q}_i$ is a principal two-sided ideal of $\OO(E,G)$. 
\end{lemma}
\begin{proof}
    For simplicity, we begin with the case $J = \{1\}$ and $e = 0$. First suppose $\mathfrak{q}_1$ is principal with generator $\psi\in\OO(E,G)$. Then, $\psi(G)\subseteq G$ and $\deg\psi = Nrd(\mathfrak{q}_1) = q_1$. Since $q_1$ is prime, either $\psi(G_1) = G_1$ or $\psi(G_1) = O_E$. By theory of hereditary orders, $\mathfrak{q}_1^2 = (q_1)$, so $[q_1] = \mu\circ\psi^2$ for some $\mu\in\Aut(E)$, and we must have $\psi(G_1) = O_E$. Since $\ker\varphi_1\subseteq\ker\psi$ and $\deg\varphi_1 = \deg\psi$, there exists an isomorphism $\tau:E/G_1\to E$ such that $\tau\circ\varphi_1 = \psi$. Since $\psi(G_i) = G_i$ for all $i = 2,...,r$, it follows that $\tau(\varphi_1(G_i)) = G_i$ for all $i = 2,...,r$. 
    By separable degree considerations and the facts that $\psi(G_1)=O_E$ and $[q_1] = \mu\circ\psi^2$, we conclude that $\psi(E[q_1]) = G_1$. 
    By construction, $\varphi_1(E[q_1]) = \widehat{G}_1$, so $\tau(\widehat{G}_1) = G_1$. Thus, $\tau$ gives an isomorphism from $D_1(E,G)=(\varphi_1(E),\widehat{G})$ to $(E,G)$.

    For the reverse implication, assume $(E,G)\sim D_1(E,G)$, so there exists an isomorphism $\eta:E/G_1\to E$ such that $\eta(\widehat{G}) = G$. Additionally, $\eta\circ\varphi_1$ is an endomorphism of $E$ and $\eta\circ\varphi_1(G) = G_2\oplus\cdots\oplus G_r\subseteq G$, so $\eta\circ\varphi_1\in\OO(E,G)$. Thus, $\eta\circ\varphi_1$ generates a two-sided ideal of $\OO(E,G)$ of norm $q_1$. To see that this is the unique two-sided ideal $\mathfrak{q}_1$ of $\OO(E,G)$ above $q_1$, we show that $[q_1]\in(\eta\circ\varphi_1)^2$:
    \begin{equation*}
    \begin{split}
        (\eta\circ\varphi_1\circ\eta\circ\varphi_1)(E[q_1]) &=(\eta\circ\varphi_1\circ\eta)(\widehat{G}_1) \\
        &=(\eta\circ\varphi_1)(G_1)\\
        &= O_E.
    \end{split}
    \end{equation*}

    The above argument establishes that $(E,G)\sim D_1(E,G)$ if and only if $\mathfrak{q}_1$ is a principal two-sided ideal of $\OO(E,G)$. 

    Now, suppose $e = 1$. For simplicity, assume $J = \varnothing$. If $\mathfrak{p}$ is a principal two-sided ideal of $\OO(E,G)$, then by \cite[23.3.19]{QASVOIGHT}, $\mathfrak{p}$ is the unique prime two-sided ideal of $\OO(E,G)$ above $p$ and $\mathfrak{p}^2 = (p)$. Let $\pi\in\OO(E,G)$ denote the generator of $\mathfrak{p}$, which must be of degree $p$. Recall that $E$ is supersingular, and so $[p]$ is purely inseparable. By definition, $\pi(E) = E$ and $\pi(G)\subseteq G$. It follows by degree argument that $\pi = \alpha\circ\pi_p^E$ for an isomorphism $\alpha:E^p\to E$ sending $G^p\to G$, so we have $(E,G)\sim F_p(E,G)$. Conversely, suppose $(E,G)\sim F_p(E,G)$. Let $\omega:E^p\to E$ be the isomorphism such that $\omega(G^p) = G$. Then $\omega\circ\pi_p^E$ generates a principal two-sided ideal of $\OO(E,G)$ of norm $p$, which is necessarily the unique principal two-sided ideal of $\OO(E,G)$ above $p$.

    For general $J,e$, both directions of the proof require us to establish the existence of an element $\psi\in\OO(E,G)$ of norm $p^e\prod_{i\in J}q_i$ with inseparable degree $p^e$ and $\psi(G_i) = O_E$ and $\psi(E[q_i]) = G_i$ for $i\in J$, and $\psi(G_i)=G_i$ for $i\not \in J$. Such properties are necessary and sufficient to prove the existence of a unique isomorphism $\eta$ such that $\eta\circ\varphi_J\circ(\pi_p^E)^e) = \psi$, and $\eta$ establishes the isomorphism required to show $(E,G)\sim D_J(F_p^e(E,G))$.
\end{proof}

\begin{theorem}\label{thm:fiber_aboveO(E,G)}
Let $\OO$ be an Eichler order of $B_{p,\infty}$ of squarefree level $N$. The size of the two-sided ideal class group of $\OO$ is equal to the number of distinct pairs $(E,G)$ of $\EGSet$ for which $\OO(E,G)\cong \OO$. Moreover, the fiber above $\OO(E,G)$ along the map $\OO(\cdot,\cdot)$ contains precisely the equivalence classes $\ec{E,G}$, $D_J\ec{E,G}$, $F_p\ec{E,G}$, and $D_J(F_p\ec{E,G})$, where $D_J$ denotes the composition of dualizing involutions for some $J\subseteq\{1,\dots,r\}$. This fiber is of size $2^{k}$ for some $k\in\{0,1,...,r+1\}$. 
\end{theorem}
\begin{proof}
By Proposition~\ref{prop:Eichler_order_correspond}, $\OO\cong \OO(E,G)$ for some $(E,G)\in\EGSet$. By Lemma~\ref{lem:D_act_triv} and Lemma~\ref{lem:Fp_act_triv}, the fiber above $\OO$ along $\OO(\cdot,\cdot)$ contains $D_J(E,G)$ and $D_J(E^p,G^p)$ for all $J\subseteq\{1,...,r\}$: $\OO\cong \OO(D_J(E,G))\cong \OO(D_J(E^p,G^p))$. Lemma~\ref{lem:sameendring} shows these are the only possible elements of the fiber above $\OO$ along $\OO(\cdot,\cdot)$. Lemma~\ref{lem:twosidedideals} shows that the two-sided ideal class group of $\OO(E,G)$ dictates when $F_p^{e_1}(D_{J_1}(E,G))\sim F_p^{e_2}(D_{J_2}(E,G))$. 
Since $\OO$ is a hereditary order of level $N$, the only possible non-principal two-sided ideals of $\OO$ are the $2^{r+1}$ products of the ideals above primes dividing $pN$, and the information of Lemma~\ref{lem:twosidedideals} completely determines the the number of distinct pairs of $\EGSet$ with isomorphic images under the map $\OO(\cdot,\cdot)$. 
Since $D_i(D_i(E,G))\sim (E,G)$ and $F_p(F_p(E,G))\sim(E,G)$, any relation among the elements of the fiber above $\OO(E,G)$ along $\OO(\cdot,\cdot)$ can be phrased as $(E,G)\sim \cdot$, and the fiber size is equal to size of the two-sided ideal class group of $\OO(E,G)$.
\end{proof}

For example, if $\mathfrak{q}_1$ is the only principal two-sided ideal of $\OO(E,G)$, then $(E,G)\sim D_1(E,G)$ and the fiber above $\OO$ along $\OO(\cdot,\cdot)$ is of size $2^{r}$ containing the inequivalent elements $F_p^e(D_J(E,G))$ for all $(e,J)\subseteq \{0,1\}\times\mathcal{P}(\{2,...,r\})$, where $\mathcal{P}$ denotes the power set. 

\begin{example}\label{ex:p1123N235}
We fix $p = 1123$ and $N = 2\cdot 3\cdot 5$ and use Magma \cite{magma} to enumerate the isomorphism classes of Eichler orders of level $N$ in the definite quaternion algebra ramified at $p$ and $\infty$. There are 456 such classes. Then, we can use built-in functions of Magma to determine the sizes of the two-sided ideal class groups of these orders:
\begin{itemize}
    \item No isomorphism classes have two-sided ideal class groups of sizes 1 or 2.
    \item Three isomorphism classes have two-sided ideal class groups of size 4.
    \item Sixty-six isomorphism classes have two-sided ideal class groups of size 8.
    \item The remaining three hundred and eighty-seven isomorphism classes have two-sided ideal class groups of 16.
\end{itemize}
This indicates the expected ``generic" behavior when $p$ is large relative to $N$: the fiber size is in general as large as possible, as collisions via the involutions we have defined do not often occur.
\end{example}

\section{Counting \texorpdfstring{$N$}{N}-isogenies \texorpdfstring{$E\to E^{p}$}{}}\label{sec:countingNisogconjs}
In this section, we require $N$ to be a prime $N\neq p$, and we will note where we use the fact that $N<<p$. We apply the results of Section~\ref{sec:vertices_to_EOs} to provide a new approximate upper bound on the number of $N$-isogenies between pairs of distinct supersingular elliptic curves with conjugate $j$-invariants. We contrast this to approximate counts and bounds provided in \cite[Lemma 6]{CGL06} and \cite[Theorem 3.9]{EHLMP_lisogconjcount} in Section~\ref{sssec:compare_to_other_bounds}.

Let $\alpha_N$ denote the number of pairs $(E,\psi)\in \EGSet$ where $\psi:E\to E^p$ is a degree-$N$ isogeny from $E$ to its $p$-power Frobenius conjugate $E^p$. Let $\alpha_N'\leq\alpha_N$ be the count of the subset of pairs $(E,\psi)$ as above with $E$ defined over $\Fpp\setminus\Fp$.

\subsection{An Approximate Upper-Bound From Eichler Orders}\label{ssec:UBfromEOs}

For this section, we assume $p\equiv1\pmod{12}$ unless otherwise stated. This is to avoid the case where there are fewer than $(N+1)$ level-$N$ structures on a supersingular elliptic curve $E$ (see Section~\ref{ssec:pnot1mod12}). In this case, we have $\#\EGSet = (N+1)(\#\mathbf{S}_p)$, where $\#(\mathbf{S}_p)$ denotes the number of supersingular $j$-invariants over $\overline{\mathbb{F}}_p$. 

Let $\mathcal{F}_{1},\mathcal{F}_{2},\mathcal{F}_{4}$ denote the number of isomorphism classes of Eichler orders whose fibers along $\OO(\cdot,\cdot)$ are sizes 1, 2, and 4, respectively.
Let $T$ denote the number of isomorphism classes of Eichler orders of level $N$ of $B_{p,\infty}$. 

\begin{proposition}\label{prop:CountingFacts} The following relations hold between the quantities $\mathcal{F}_1,\mathcal{F}_2,\mathcal{F}_4,T,(\#\mathbf{S}_p)$:
\begin{equation}\label{eq:FibersGiveTypeNumb}
\mathcal{F}_1 + \mathcal{F}_2 + \mathcal{F}_4 = T
\end{equation}
\begin{equation}\label{eq:SizeOfEGSetFromFibers}
\mathcal{F}_1 + 2\mathcal{F}_2 + 4\mathcal{F}_4 =\#\EGSet =  (N+1)(\#\mathbf{S}_p).
\end{equation}
\end{proposition}
\begin{proof}
For \eqref{eq:FibersGiveTypeNumb}: every Eichler order of level $N$ has fiber size 1, 2, or 4 along $\OO(\cdot,\cdot)$ by Theorem~\ref{thm:fiber_aboveO(E,G)}. For \eqref{eq:SizeOfEGSetFromFibers}: the size of $\EGSet$ is $(N+1)(\#\mathbf{S}_p)$, and every element of $\EGSet$ lies in a fiber above some isomorphism class of Eichler order of level $N$ along $\OO(\cdot,\cdot)$.
\end{proof}

Combine the equations in Proposition~\ref{prop:CountingFacts} to solve for $\mathcal{F}_2$:
\begin{equation}\label{eq:CombineEqs}
    \mathcal{F}_2 = 2T - \frac{N+1}{2}(\#\mathbf{S}_p) - \frac{3}{2}\mathcal{F}_1.
\end{equation}
When the fiber above $\OO(E,G)$ is size one, all of the involutions described in Section~\ref{sec:vertices_to_EOs} act as the identity: In particular, the isogeny $\varphi_G$ is an endomorphism with kernel stable under the $p$-power Frobenius. 
For $N$ much smaller than $p$, curves with degree-$N$ endomorphisms are rare: The number of curves with a non-scalar endomorphism of degree less than or equal to $N$ is $O(N^{3/2})$ (see for example \cite{BonehLove}). Example~\ref{ex:p1123N235} provides some empirical evidence to this claim.
Assuming $\mathcal{F}_1 = 0$, we obtain an approximate upper bound for $\mathcal{F}_2$ from Equation~\eqref{eq:CombineEqs}:
\begin{equation}\label{eq:approx_ub}
2T - \frac{N+1}{2}(\#\mathbf{S}_p).
\end{equation}
The number of degree-$N$ isogenies between conjugate curves $\alpha_N'$ is also generically counted in $\mathcal{F}_2$. To see this, begin by noting that if $E$ is defined over $\Fpp$ and not $\Fp$, then the fiber above $\OO(E,G)$ is either of size 2 or 4. Furthermore, if $G$ is the kernel of an $N$-isogeny from $E$ to $E^p$, then $E/G = E^p$.  From here, either $G^p = \hat{G}$ or $G^p\neq \hat{G}$. In the first case, the fiber above $\OO(E,G)$ is size 2. If $G^p \neq \hat{G}$, then there are two separate degree-$N$ isogenies from $E^p$ to $E$: one with kernel $G^p$ and the other with kernel $\hat{G}$. This corresponds to a double-edge in the $N$-isogeny graph, which is a rare occurrence (as discussed in \cite{arpin2019adventures}). Therefore, Equation \eqref{eq:approx_ub} gives an approximate upper-bound for the number of conjugate pairs of supersingular curves over $\Fpp\setminus\Fp$ connected by an $N$-isogeny, namely $\alpha_N'/2$. We conclude that an approximate upper-bound for $\alpha_N'$ is:
\begin{equation}\label{eq:approx_ub_actual}
    4T - (N+1)(\#\mathbf{S}_p).
\end{equation}

This approximate upper-bound for $\alpha_N'$ can be compared with the data computed in \cite{arpin2019adventures}, which we discuss in Section~\ref{sssec:compare_to_other_bounds}.

\subsection{Comparison With Other Bounds}\label{sssec:compare_to_other_bounds} 

The author's first interest in the question of counting $N$-isogenous conjugate curves began with research as part of the collaborative work in \cite{arpin2019adventures}. We wished to identify the frequency of \textit{mirror paths}, which are invariant under the Frobenius conjugate. These mirror paths necessarily have a central point of symmetry, which either corresponds to a $j$-invariant defined over $\Fp$, or a pair of $N$-isogenous conjugate $j$-invariants both defined over $\Fpp$ and not $\Fp$. We posed a question about counting the number of $N$-isogenous conjugate pairs, as described in the second mirror path scenario. This corresponds to estimating $\alpha_N'$. In \cite{arpin2019adventures}, we computed $\alpha'_N$ for a wide range of values $p$.

Subsequently, \cite{EHLMP_lisogconjcount} considered the question of counting the number $\alpha_N$ of supersingular $j$-invariants with an $N$-isogeny to their $p$-power Frobenius conjugate. They pointed out that an upper-bound for this value, which they denote $|S^p|$, could be computed using \cite[Lemma 6]{CGL06}, which provides an approximation for this value. The authors also provided a lower-bound \cite[Theorem 3.9]{EHLMP_lisogconjcount}:
\[|S^p|\geq \frac{\sqrt{Np}}{6(N+1)\log\log(Np)}.\]

This lower-bound is an easily computed function which provides a lower-bound on the class number of the order $\ZZ[\sqrt{-Np}]$. In Figure~\ref{fig:3IsogConj}, we plot this rational function for $N=3$ against the data for $3$-isogenous conjugates provided in \cite{arpin2019adventures} and the upper-bound for $\alpha_3'$ in Equation~\eqref{eq:approx_ub_actual} is plotted. 

%\begin{figure}
%    \centering
%    \includegraphics[scale=.55]{Images/Rplot.png}
%    \caption{For primes between 5 and 100000, we plot the actual count $\alpha_3'$ of $3$-isogenies between distinct $p$-power Frobenius conjugate curves, the approximate upper bound for $\alpha_3'$ provided in Equation~\eqref{eq:approx_ub_actual}, and the lower bound for $\alpha_3$ provided by \cite{EHLMP_lisogconjcount}}
%    \label{fig:3IsogConj}
%\end{figure}

\begin{figure}
    \centering
    \begin{subfigure}[b]{0.49\textwidth}
        \includegraphics[scale=.45]{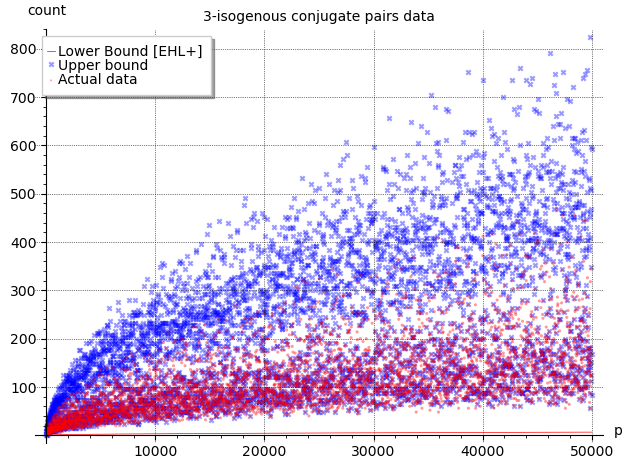}
        \caption{All primes $5\leq p<50000$. A total of $5131$ data points.}\label{fig:part1}
    \end{subfigure}
    \hfill 
    \begin{subfigure}[b]{0.49\textwidth}
        \includegraphics[scale=.45]{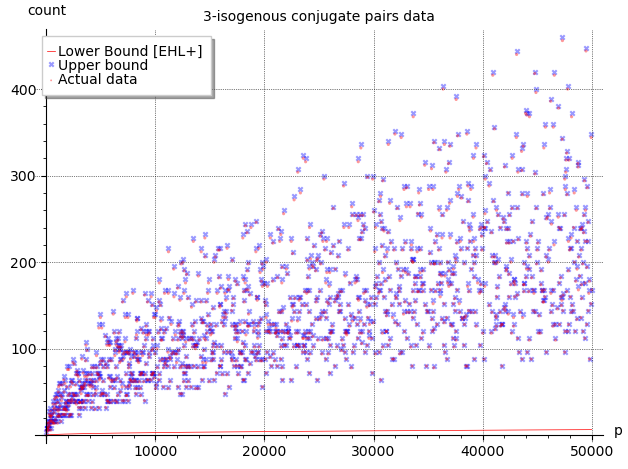}
        \caption{Primes $5\leq p<50000$ with $p\equiv1\pmod{12}$. A total of $1264$ data points.}
    \end{subfigure}
    \caption{The counts of 3-isogenies between distinct $p$-power Frobenius conjugate supersingular elliptic curves over $\Fpbar$. For $p\equiv 1\pmod{12}$, the estimated upper bound described in this section was never off by more than four: exactly accurate for $23.33\%$ of the data, over by two for $51.34\%$ of the data, and over by four for $25.32\%$ of the data. The lower bound from \cite{EHLMP_lisogconjcount} is plotted in red.}
    \label{fig:3IsogConj}
\end{figure}

The big-O notation approximation provided in \cite[Lemma 6]{CGL06} can be adjusted to give an exact count of $N$-isogenous conjugate curves. To begin this analysis, we provide the statement of \cite[Lemma 6]{CGL06}:
\begin{lemma}[Lemma 6~\cite{CGL06}]\label{lem:CGL_NIsogConj}
Let $i$ be a non-negative integer. The number $\alpha(i)$ of supersingular $j$-invariants such that $\text{dist}_G(j,j^p)\leq i$ is the number of pairs $(E,g)$ consisting of a supersingular elliptic curve $E$ and an endomorphism $g$ of $E$ of degree $p\cdot\ell^j$, $j\leq i$, up to isomorphism. Assume that $i\leq \log_\ell(p/4)$. Then
\[\alpha(i) = \ell^{i/2}\widetilde{O}(\sqrt{p}).\]
\end{lemma}

Inspired by Lemma~\ref{lem:CGL_NIsogConj}, we prove the precise value of $\alpha_N$. Chenu and Smith \cite[Theorem 2]{chenu2021higherdegree} provide an alternative proof of this proposition.

\begin{proposition}\label{prop:NIsogConj}
Let $N$ be a prime such that $N<p/4$. The value $2\alpha_N$ is equal to the number of pairs consisting of a supersingular elliptic curve $E$ and an embedding $\ZZ[\sqrt{-pN}]$ into $\End(E)$. 

Furthermore, 
\[2\alpha_N = \begin{cases}
\mid\mathcal{Cl}(\ZZ[\frac{1 + \sqrt{-pN}}{2}])\mid + \mid\mathcal{Cl}(\ZZ[\sqrt{-pN}])\mid &\text{ if } -pN\equiv 3\pmod{4}\\
\mid\mathcal{Cl}(\ZZ[\sqrt{-pN}])\mid &\text{ if } -pN\equiv 1\pmod{4}
\end{cases}\]
\end{proposition}
The factor of two appears because two embeddings which differ by a factor of $-1$ on the generator $\sqrt{-pN}$ are counted as distinct, whereas the two isogenies $\psi,-\psi$ are not considered distinct.
\begin{proof}
By definition the number $\alpha(1)$ counts pairs $(E,h)$ where $E$ is a supersingular elliptic curve and $h:E^p\to E$ is a degree-$N$ isogeny between $E$ and its conjugate $E^p$. Every endomorphism of $E$ can be factored into separable and purely inseparable parts. In particular, every endomorphism $g$ of $E$ of degree $pN$ can be factored uniquely into $h\circ\pi_p$, where $\pi_p$ is the $p$-power Frobenius map and $h$ is an isogeny of degree $N$.
\begin{center}
\begin{tikzcd}
E\arrow[r,"\pi_p"]\arrow[rr,bend right,swap,"g"]& E^p \arrow[r,"h"]& E
\end{tikzcd}
\end{center}
The data $(E,g)$ is equivalent to the data $(E,h)$. To count pairs $(E,g)$, we are looking to count embeddings of $\ZZ[\sqrt{-pN}]$ into $\End(E)$ \cite[Lemma 6]{CGL06}. The action of the class group of $\ZZ[\sqrt{-pN}]$ is free and transitive on a subset of the primitively $\ZZ[\sqrt{-pN}]$-oriented supersingular elliptic curves, by \cite{onuki2021}. By \cite[Theorem 2]{chenu2021higherdegree}, this subset actually contains all primitively $\ZZ[\sqrt{-pN}]$-oriented supersingular elliptic curves. By this free and transitive action, the number of such embeddings is equal to the class number of $\ZZ[\sqrt{-pN}]$. The number of primitive embeddings is $\ZZ[\sqrt{-pN}]$, but if $\ZZ[\sqrt{-pN}]$ is properly contained in the maximal order of $\QQ(\sqrt{-pN})$, then this is not the full picture. If the ring of integers $\mathcal{O}_{\QQ(\sqrt{-pN})} = \ZZ[\frac{1 + \sqrt{-pN}}{2}]\supsetneq\ZZ[\sqrt{-pN}]$, then we will also want to count primitive embeddings of $\mathcal{O}_{\QQ(\sqrt{-pN})}$. 

The total number of embeddings (and thus, $N$-isogenies to a conjugate curve) is:
\[ \begin{cases}
\mid\mathcal{Cl}(\ZZ[\frac{1 + \sqrt{-pN}}{2}])\mid + \mid\mathcal{Cl}(\ZZ[\sqrt{-pN}])\mid &\text{ if } -pN\equiv 3\pmod{4}\\
\mid\mathcal{Cl}(\ZZ[\sqrt{-pN}])\mid &\text{ if } -pN\equiv 1\pmod{4}
\end{cases}\]
Note that this count includes embeddings which differ by an automorphism of the field, in particular the automorphisms $\pm1$. Since the field $\QQ(\sqrt{-pN})$ is quadratic, we always have Galois group $\cong \ZZ/2\ZZ$. The endomorphisms corresponding to $\pm\sqrt{-pN}$ are not distinct, so we divide this embedding count by two to get the number of pairs $(j(E),\psi)$ where $\psi:E\to E^p$ is an $N$-isogeny. This divided count is what we see if we look at the supersingular $N$-isogeny graph. 
\end{proof}

\begin{figure}
     \centering
     \begin{subfigure}[b]{0.49\textwidth}
         \centering
        \begin{tikzpicture}
\node[black, circle, draw] (alpha) at (0,0) {$\alpha$};
\node[black, circle, draw] (alphabar) at (2,0) {$\overline{\alpha}$};
\node[black,circle,draw] (41) at (0,2) {$41$};
\node[black,circle,draw] (9) at (1,4) {$9$};
\node[black,circle,draw] (50) at (2,2) {$50$};

\draw[-,loop left] (9) to (9);
\draw[-,loop above] (9) to (9);
\draw[-] (9) to (41);
\draw[-] (9.west) to (41.north);

\draw[-] (41) to (alphabar);
\draw[-] (41) to (alpha);

\draw[-] (50) to (alphabar);
\draw[-] (50) to (alpha);

\draw[-,loop right] (50) to (50);
\draw[-,loop above] (50) to (50);

\draw[-] (alpha) to (alphabar);
\draw[-] (alpha) to [out=330,in=210] (alphabar);

\end{tikzpicture}
         \caption{$p = 61$, supersingular $N=3$-isogeny graph}
         \label{subfig:p61N3}
     \end{subfigure}
     \hfill
     \begin{subfigure}[b]{0.49\textwidth}
         \centering
        \begin{tikzpicture}
\node[black, circle, draw] (alpha) at (0,0) {$\alpha$};
\node[black, circle, draw] (alphabar) at (2,0) {$\overline{\alpha}$};
\node[black,circle,draw] (41) at (0,2) {$41$};
\node[black,circle,draw] (9) at (2,2) {$9$};
\node[black,circle,draw] (50) at (1,4) {$50$};

\draw[-,loop left] (50) to (50);
\draw[-,loop above] (50) to (50);
\draw[-] (50) to (41);
\draw[-] (50.east) to [out=30,in=60] (9);
\draw[-] (50.east) to [out=330,in=90] (9);
\draw[-] (50.east) to (9);

\draw[-] (41) to (9);
\draw[-] (41) to (alphabar);
\draw[-] (41.east) to (alphabar.north);
\draw[-] (41) to [out=240,in=120] (alpha);
\draw[-] (41) to [out=300,in=60] (alpha);

\draw[-] (9) to (alphabar);
\draw[-] (9) to (alpha);

\draw[-] (alpha) to (alphabar);
\draw[-] (alpha) to [out = 330,in=210] (alphabar);
\draw[-] (alpha) to [out = 300,in=240] (alphabar);

\end{tikzpicture}
         \caption{$p=61$, supersingular $N=5$-isogeny graph}
         \label{subfig:p61N5}
     \end{subfigure}
        \caption{Illustrative examples for Proposition~\ref{prop:NIsogConj}.}
        \label{fig:twographs}
\end{figure}

\begin{example}[$-pN\equiv 1\pmod{4}$]
Let $p = 61$, $N= 3$. By Proposition~\ref{prop:NIsogConj},
\[\alpha_N = \frac{1}{2}|\mathcal{Cl}(\ZZ[\sqrt{-61\cdot3}])| + \frac{1}{2}\left|\mathcal{Cl}\left(\ZZ\left[\frac{1+\sqrt{-61\cdot3}}{2}\right]\right)\right| = \frac{8}{2} + \frac{8}{2} = 8.\]
We provide the supersingular $3$-isogeny graph over $\overline{\mathbb{F}}_{61}$ in Figure~\ref{subfig:p61N3}. Since $61\equiv1\pmod{12}$, this graph can be presented as undirected by identifying isogenies and their duals. We see the eight $3$-isogenies to a conjugate curve:
\begin{itemize}
    \item Two 3-isogenies $E_{50}\to E_{50}$
    \item Two 3-isogenies $E_{9}\to E_{9}$
    \item Two 3-isogenies $E_{\alpha}\to E_{\overline{\alpha}}$
    \item Two 3-isogenies $ E_{\overline{\alpha}}\to E_{\alpha}$
\end{itemize}
\end{example}

\begin{example}[$-pN\equiv 3\pmod{4}$]
Let $p = 61$, $N=5$. By Proposition~\ref{prop:NIsogConj}, 
\[\alpha_N = \frac{1}{2}|\mathcal{Cl}(\ZZ[\sqrt{-61\cdot5}])| = \frac{16}{2} = 8.\]
We provide the supersingular $5$-isogeny graph over $\overline{\mathbb{F}}_{61}$ in Figure~\ref{subfig:p61N5}. Since $61\equiv1\pmod{12}$, this graph can be presented as undirected by identifying isogenies and their duals. We see the eight $5$-isogenies to a conjugate curve:
\begin{itemize}
    \item Two 5-isogenies $E_{50}\to E_{50}$
    \item Three 5-isogenies $E_{\alpha}\to E_{\overline{\alpha}}$
    \item Three 5-isogenies $ E_{\overline{\alpha}}\to E_{\alpha}$
\end{itemize}

\end{example}

\section{The Category \texorpdfstring{$\mathcal{S}_N$}{}}\label{sec:category}
The Deuring correspondence was rephrased as a categorical equivalence by Kohel \cite{Kohel}. In this categorical version, the supersingular elliptic curve objects are enhanced by the data of a Frobenius isogeny, as are the quaternion objects. Voight \cite{QASVOIGHT} presents a variation of this equivalence of categories, without this additional enhancement:

\begin{theorem}[Theorem 42.3.2 \cite{QASVOIGHT}]
Fix a supersingular elliptic curve $E_0$ with endomorphism ring $\OO_0$. The functor $\Hom(\cdot,E_0)$ defines an equivalence of categories from the category of isomorphism classes of supersingular elliptic curves over $\Fpbar$ under isogenies and the category of invertible left-$\OO_0$-modules, under nonzero left $\OO_0$-module homomorphisms.
\end{theorem}

We present an equivalence of categories for supersingular elliptic curves with level-$N$ structure. We return to the case where $N$ is any squarefree integer coprime to $p$.

\subsection{Equivalence of Categories}\label{ssec:equiv_of_cats}
We begin by defining the categories in question:
\begin{definition}[Supersingular elliptic curves with level-$N$ structure]
Let $\mathcal{S}_N$ denote the category with objects given by pairs $(E_1,G_1)$, where $E_1$ is a supersingular elliptic curve over $\overline{\mathbb{F}}_p$ up to $\overline{\mathbb{F}}_p$-isomorphism and $G_1\subset E_1[N]$ is fixed order-$N$ subgroup. A morphism between two objects $(E_1,G_1)$ and $(E_2, G_2)$ is an isogeny $\psi:E_1\to E_2$ such that $\psi(G_1)\subseteq G_2$.
\end{definition}

 We fix $(E,G)\in\mathcal{S}_N$ and the Eichler order $\OO(E,G)$ for the remainder of this section. Define the following category:
 
\begin{definition}[Invertible left $\OO(E,G)$-modules]
Let $\mathcal{LM}$ denote the category with objects invertible left $\OO(E,G)$-modules. A morphism between objects is given by a left $\OO(E,G)$-module homomorphism.
\end{definition}
It is straightforward to check that these are well-defined categories.

\begin{definition}
We let $\hh{(E,G)}$ denote the functor $\Hom_{\mathcal{S}_N}(-,(E,G))$, so 
\[\hh{(E,G)}(E',G') = \Hom_{\mathcal{S}_N}((E',G'),(E,G)).\]
On morphisms, $\hh{(E,G)}$ maps $f:(E_0,G_0)\to (E_1,G_1)$ to the morphism $\Hom((E_1,G_1),(E,G))\to \Hom((E_0,G_0),(E,G))$ given by $g\mapsto g\circ f$. 
\end{definition}

\begin{theorem}[Equivalence of Categories]\label{thm:cats_equiv}
Fix a supersingular elliptic curve $E$ defined over $\overline{\mathbb{F}}_p$ and a cyclic subgroup $G\subset E[N]$ of order $N$. The functor $\hh{(E,G)}$ is a contravariant functor from the category $\mathcal{S}_N$ to the category $\mathcal{LM}$. This functor defines an equivalence of categories.
\end{theorem}

\begin{proof}
Lemma~\ref{lem:functor} shows that $\hh{(E,G)}$ is well-defined as a functor. To see that $\hh{(E,G)}$ defines an equivalence of categories, it remains to show that $\hh{(E,G)}$ is essentially surjective and fully faithful. 

First, we show that $\hh{(E,G)}$ is essentially surjective. Consider the objects $I$ of $\mathcal{LM}$: Since $I$ is an invertible left $\OO(E,G)$-module, it is a rank 1 $\OO(E,G)$-module. $\OO(E,G)$ is rank 4 over $\ZZ$, so $I$ is also rank 4 over $\ZZ$. Since $I$ is a rank 1 $\OO(E,G)$-module it is locally isomorphic to $\OO(E,G)$. This local isomorphism extends to an isomorphism $I\otimes_\ZZ\QQ\cong\OO(E,G)\otimes_\ZZ\QQ$, which gives an inclusion of $I$ into $B_{p,\infty}^E:=\OO(E,G)\otimes_\ZZ \QQ$. By \cite[Theorem 9.3.6]{QASVOIGHT} $I$ is a fractional ideal of $\OO(E,G)$ in $B_{p,\infty}^E$. Scaling by an integer prime to $N$, we can assume $I$ is an integral left ideal of $\OO(E,G)$ of norm $\Nrd(I)$. The norm $\Nrd(I)$ must be prime to $N$, otherwise it would violate the invertibility of the original left $\OO(E,G)$-module. By Lemma~\ref{lem:ideals_and_isogenies}, $I = \hh{(E,G)}(E_I,G_I)\varphi_I$ such that $\varphi_I:E\to E_I$ is a degree $\Nrd(I)$ isogeny and $\varphi_I(G) \subseteq G_I$. This identification shows $\hh{(E,G)}$ is essentially surjective.

Lastly, we show that $\hh{(E,G)}$ is fully faithful. In particular, we need to show that the map
\begin{equation*}
\begin{split}
    \Hom_{\mathcal{S}_N}((E_{I'},G_{I'}),(E_I,G_I)) &\to\Hom_{\mathcal{LM}}(\hh{(E,G)}(E_I,G_I),\hh{(E,G)}(E_{I'},G_{I'}))
\end{split}
\end{equation*}
from morphisms in $\mathcal{S}_N$ to morphisms in $\mathcal{LM}$ is bijective. This is accomplished in Lemma~\ref{lem:FullyFaithful}.

\end{proof}

\begin{lemma}\label{lem:functor}
Let $(E',G')\in \mathcal{S}_N$. Then, $\hh{(E,G)}(E',G')$ is a $\mathbb{Z}$-module of rank 4 that is invertible as a left $\OO(E,G)$-module under post-composition.
\end{lemma}

\begin{proof} 

By \cite[Lemma 42.1.11]{QASVOIGHT}, $\Hom(E',E)$ is a rank 4 $\ZZ$-module.
By definition $\hh{(E,G)}(E',G')$ contains the set $\{\phi\in\Hom(E',E):G'\subseteq\ker\phi\}$.
By Corollary III.4.11 of \cite{AEC}, this set is equivalently characterized:
\[\{\phi\in \Hom(E',E):G'\subseteq \ker\phi\} = \{\phi\circ\varphi_{G'}: \phi\in \Hom(E'/G',E)\} = \Hom(E'/G',E)\varphi_{G'},\]
where $\varphi_{G'}:E'\to E'/G'$ is the unique separable isogeny with $\ker(\varphi_{G'}) = G'$. This is an ideal of $\Hom(E'/G',E)$, which is rank 4. It follows that each $\hh{(E,G)}(E',G')$ is rank 4 as well.  

To prove the invertibility of $\hh{(E,G)}(E',G')$ as a left $\OO(E,G)$-module, we use the fact that $\OO(E,G)$ is isomorphic to a hereditary order of squarefree level coprime to $p$ in the quaternion algebra $B_{p,\infty}$, see Theorem~\ref{thm:O(E,G)_is_EO}. This strategy is similar to the maximal order case, detailed in \cite[Lemma 42.1.11]{QASVOIGHT}. Take a nonzero isogeny $\psi\in \hh{(E,G)}(E',G')$ and let $\widehat{\psi}$ denote the dual of $\psi$. Then, $I:=\hh{(E,G)}(E',G')\widehat{\psi}\subset \OO(E,G)$ is an integral left $\OO(E,G)$ ideal, and is thus invertible by the hereditary property of $\OO(E,G)$ (all lattices of hereditary orders are invertible by \cite[Section 23.1.2]{QASVOIGHT}). The same holds for $\hh{(E,G)}(E',G')$ as a left $\OO(E,G)$-module.
%Invertibility as a left $\OO(E,G)$-module  follows from the association of $\OO(E,G)$ to a hereditary order of the quaternion algebra $B_{p,\infty}$, which we prove in Theorem~\ref{thm:O(E,G)_is_EO}. Hereditary orders have the property that all lattices $I\subset B_{p,\infty}$ with hereditary left or right order are invertible (Section 23.1.2 of \cite{QASVOIGHT}).
\end{proof}

\begin{lemma}\label{lem:ideals_and_isogenies}
Fix an integral left $\OO(E,G)$-ideal $I$ of norm prime to $N$. There exists an isogeny $\varphi_I:E\to E_I$ and a subgroup $G_I\subseteq E_I[N]$ of order $N$ such that $\varphi_I(G)\subseteq G_I$, and $I = \hh{(E,G)}(E_I,G_I)\varphi_I$, and $\Nrd(I) = \deg(\varphi_I)$. 
\end{lemma}

\begin{proof}
By Theorem~\ref{thm:O(E,G)_is_EO}, $\OO(E,G)$ is isomorphic to an Eichler order. It is contained in the maximal order $M$ isomorphic to $\End(E)$. By Proposition~\ref{prop:EOidealsandMOideals} (and \cite[Lemma 3]{SqiSign}), the integral left ideals of the Eichler order $\OO(E,G)$ of norm prime to $N$ are in bijection with the integral left ideals of the maximal order $\End(E)\supset\OO(E,G)$ of norm prime to $N$. This bijection sends the integral left ideal $I$ of $\OO(E,G)$ to  $\End(E)I$. To avoid confusion, we will write $\End(E)I$ when we mean the left ideal of $\End(E)$, but use $I$ when we are referring to $I$ as a left $\OO(E,G)$-ideal. As a left integral ideal of $\End(E)$, $\End(E)I$ can be used to define an isogeny in the following way (see \cite[Section 42.2]{QASVOIGHT}). Let
\begin{equation}\label{eq:ideal_to_varphi}
  E[\End(E)I] := \bigcap_{\alpha\in \End(E)I}\ker(\alpha)  
\end{equation}
be the scheme theoretic intersection, and define $\varphi_I:E\to E_I =: E/E[\End(E)I]$ via $\ker\varphi_I = E[\End(E)I]$. By \cite[Proposition 42.2.16]{QASVOIGHT}, $\deg\varphi_I = \Nrd(\End(E)I)$. Since $\End(E)I$ is of norm prime to $N$, $\varphi_I$ maps $G\subset E[N]$ to some $G_I\subset E_I[N]$.
\end{proof}

\begin{lemma}\label{lem:objects_of_SN_(EI,GI)}
Every object $(E',G')$ of $\mathcal{S}_N$ is of the form $(E_I,G_I)$ for some integral left $\OO(E,G)$-ideal $I$, where $I$ can be chosen to have norm prime to $N$.
\end{lemma}
\begin{proof}
%\textcolor{red}{Another way to see this? Strong approximation or something?}
Let $\ell$ be a prime such that $\ell\nmid pN$. By the connectedness of the $\ell$-isogeny graph $\EpCat$ of supersingular elliptic curves with level-$N$ structure (Theorem~\ref{thm:Connectedness}), there exists a chain of $\ell$-isogenies connecting the vertices $(E,G)$, $(E',G')$. Let $\varphi:E\to E'$ denote this isogeny composition, where $\varphi(G) = G'$. By the theory described in Section~\ref{ssec:bckgrnd_isogenies}, the kernel of $\varphi$ corresponds to an integral left-$\End(E)$ ideal $I_\varphi$ of norm equal to the degree of $\varphi$, which is a power of $\ell$ by construction and thus is coprime to $N$. Since the codomain of $\varphi$ is $E'$, we have $E' = E_{I_\varphi}$. Since $\varphi_I(G)\subseteq G'$ and the degree of $\varphi_I$ is coprime to $N$, we have $\varphi_I(G) = G'$ and furthermore $G' = G_{I_\varphi}$. By the bijection in Lemma~\ref{prop:EOidealsandMOideals}, $I\cap \OO(E,G)$ is an integral left ideal of the Eichler order $\OO(E,G)$.
\end{proof}

%\begin{lemma}\label{lem:HomsOfInvLineBunds}
%For any invertible left $\OO$-ideals $I, I'$, $\Hom_{\mathcal{LM}}(I,I')\cong I^{-1}I'$.
%\end{lemma}
%\begin{proof}
%Every element $\gamma$ of $I^{-1}I'$ can be used to define a map from $I$ to $I'$ via multiplication on the right by $\gamma$. This map is compatible with the left action of $\OO$, so $I^{-1}I'\subseteq \Hom_{\mathcal{LM}}(I,I')$. 
%\end{proof}

\begin{lemma}\label{lem:v_42.2.22}
Let $I,I'\subset\OO(E,G)$ be nonzero integral left $\OO(E,G)$-ideals of norm prime to $N$. Define $\Hom((E_I,G_I), (E,G))\Hom((E_{I'}, G_{I'}),(E_I,G_I))$ to be the collection of isogenies
\[\{\varphi:(E_{I'},G_{I'})\to (E,G)\mid \varphi = \sum_{i}\alpha_i\beta_i,\alpha_i\in\Hom((E_I,G_I),(E,G)), \beta_i\in\Hom((E_{I'},G_{I'}),(E_I,G_I))\}.\] 
Then, the natural map
\[\Hom((E_I,G_I), (E,G))\Hom((E_{I'}, G_{I'}),(E_I,G_I)) \to \Hom((E_{I'}, G_{I'}), (E,G))\]
is a left $\OO(E,G)$-module isomorphism. 
%giving a further bijection
%\[\Hom((E_{I'},G_{I'}), (E_I,G_I)) \leftrightarrow I^{-1}I',\]
%where $I^{-1} := \overline{I}\Nrd(I)^{-1}$ and $\overline{I} := \{\overline{\alpha}:\alpha \in I\}= \{\widehat{\alpha}:\alpha \in I\}$.
\end{lemma}

\begin{proof}
By construction of $\Hom((E_I,G_I), (E,G))\Hom((E_{I'}, G_{I'}),(E_I,G_I))$, the map above is injective.

By Lemma~\ref{lem:ideals_and_isogenies}, we have:
\[I = \Hom((E_I,G_I),(E,G))\phi_I\]
where $\phi_I:E\to E_I$ with $\phi_I(G) =: G_I$, and $N = \deg(\phi_I) = \Nrd(I)$. Since $\OO(E,G)$ is a hereditary order, $I$ is invertible, and by Proposition 16.6.15\cite{QASVOIGHT}, $(m) := (\Nrd(I)) = I\overline{I}$. The quaternion element $[m]$ has an expression as an element of 
\[I\overline{I} = (\Hom((E_I,G_I), (E,G))\phi_I)\overline{(\Hom((E_I,G_I), (E,G))\phi_I)}.\] 
There exist finitely many $\alpha_i,\beta_i\in \Hom((E_I,G_I),(E,G))$ to give this expression:
\begin{equation*}
    [m] = \sum_i(\alpha_i\phi_I)\widehat{(\beta_i\phi_I)} =\sum_i\alpha_i\phi_I\hat{\phi}_I\widehat{\beta}_i =[m] \sum_i\alpha_i\widehat{\beta}_i
\end{equation*}
Since each $\alpha_i\widehat{\beta}_i: (E,G)\to (E,G)$, the sum $\sum_i\alpha_i\widehat{\beta}_i\in\OO(E,G)$, and $[1] = \sum_i\alpha_i\widehat{\beta}_i$.

Take any $\psi\in \Hom((E_{I'},G_{I'}), (E,G))$. We need to show that it has a pre-image in 
$$\Hom((E_I,G_I),(E,G))\Hom((E_{I'},G_{I'}),(E_I,G_I))$$ 
under the natural map (composition and sum). To see this, post-compose $\psi$ by $\sum_i\alpha_i\widehat{\beta}_i$:
\begin{equation*}
    \psi = \sum_i\alpha_i\widehat{\beta}_i\psi = \sum_i\alpha_i(\widehat{\beta}_i\psi)
\end{equation*}
By construction, $\alpha_i\in\Hom((E_I,G_I),(E,G))$ and $\widehat{\beta}_i\psi \in\Hom((E_{I'},G_{I'}), (E_I,G_I))$, so the map \[\Hom((E_I,G_I),(E,G))\Hom((E_{I'},G_{I'}),(E_I,G_I))\to \Hom((E_{I'},G_{I'}), (E,G))\] 
is indeed surjective.

%The second bijection follows immediately from the first. To see this, follow the definitions of $I$ and $I'$ as ideals of $\OO(E,G)$:
%\begin{equation*}
%\begin{split}
%    \Hom((E_I,G_I),(E,G))\Hom((E_{I'},G_{I'}),(E,G)) &\leftrightarrow \Hom((E_{I'},G_{I'}),(E,G))\\
%    (I(\frac{1}{\deg\phi_I}\widehat{\phi_I}))\Hom((E_{I'},G_{I'}),(E,G)) &\leftrightarrow (I'(\frac{1}{\deg\phi_{I'}}\widehat{\phi}_{I'}))\\
%    \Hom((E_{I'},G_{I'}),(E,G)) &\leftrightarrow (I(\frac{1}{\deg\phi_I}\widehat{\phi_I}))^{-1}(I'(\frac{1}{\deg\phi_{I'}}\widehat{\phi}_{I'}))\\
%    \Hom((E_{I'},G_{I'}),(E,G)) &\leftrightarrow \phi_II^{-1}I'(\frac{1}{\deg\phi_{I'}}\widehat{\phi}_{I'})\\
%    \Hom((E_{I'},G_{I'}),(E,G)) &\leftrightarrow I^{-1}I'\\
%\end{split}
%\end{equation*}
\end{proof}

\begin{lemma}[Fully Faithful]\label{lem:FullyFaithful}
The functor $\hh{(E,G)}$ is fully faithful. In particular, the map
\begin{equation*}
\begin{split}
\Hom_{\mathcal{S}_N}((E_{1},G_{1}),(E_{2},G_{2}))&\to \Hom_{\mathcal{LM}}(\hh{(E,G)}(E_{2},G_{2}), \hh{(E,G)}(E_{1},G_{1}))\\
\phi&\mapsto -\circ\phi
\end{split}
\end{equation*}
from morphisms in $\mathcal{S}_N$ to morphisms in $\mathcal{LM}$ furnished by $\hh{(E,G)}$ is bijective.
\end{lemma}
\begin{proof}
    First, we check that the functor is faithful. Suppose $f,f':(E_1,G_1)\to (E_2,G_2)$ and $\hh{(E,G)}(f) = \hh{(E,G)}(f')$, that is $g\circ f = g\circ f'$ for every $g\in \Hom((E_2,G_2),(E,G))$. Take any such nonzero separable isogeny $g$ of degree coprime to $\deg f$. Notice that $\deg f = \deg f'$ by comparing degrees on the left and righthand sides of $g\circ f = g\circ f'$. Postcomposing with $\widehat{g}$, we have $[\deg g]\circ f = [\deg g]\circ f'$. Since $\Hom((E_1,G_1),(E,G))$ is torsion-free, we have $f = f'$.

To see that it is also full, let $\psi:\Hom((E_2,G_2),(E,G))\to \Hom((E_1,G_1),(E,G))$ be a nonzero left $\OO(E,G)$-module homomorphism. We need to show that there exists $f\in\Hom((E_1,G_1),(E_2,G_2))$ such that $\psi(x) = x\circ f$ for all $x\in\Hom((E_2,G_2),(E,G))$. Begin by applying Lemma~\ref{lem:objects_of_SN_(EI,GI)} to find integral left $\OO(E,G)$-ideals $I_1,I_2$ of norm prime to $N$ such that $(E_i,G_i) = (E_{I_i},G_{I_i})$ and corresponding isogenies $\varphi_{I_i}:E\to E_{I_i}$ for $i = 1,2$. 
By Lemma~\ref{lem:ideals_and_isogenies}:
\begin{equation}\label{eq:idealdesc}
\begin{split}
    \Hom((E_1,G_1),(E,G))\varphi_{I_1} &= I_1,\\
    \Hom((E_2,G_2),(E,G))\varphi_{I_2} &= I_2.
\end{split}
\end{equation}
The map $\psi$ induces a map map of ideals $\psi':I_2\to I_1$ given by $x\circ\varphi_{I_2}\mapsto \psi(x)\circ\varphi_{I_1}$. Since $\psi$ is injective, $\psi'$ is as well giving an isomorphism of left $\OO(E,G)$-ideals (namely $I_2\cong \psi'(I_2)\subseteq I_1$). 
Any such isomorphism is given by precomposition by an invertible element of the quaternion algebra, so there exists some $\beta\in\OO(E,G)\otimes_\ZZ\QQ$ such that $\psi'(x\circ\varphi_{I_2}) = x\circ\varphi_{I_2}\beta$, for all $x\circ\varphi_{I_2}\in I_2$. This shows $I_2\beta\subseteq I_1$ and $\beta\in I_2^{-1}I_1$. Together with the definition of $\psi'$ this gives 
\begin{equation}\label{eq:xI2}
    x\circ\varphi_{I_2}\beta = \psi(x)\circ\varphi_{I_1}
\end{equation}
for all $x\in\Hom((E_2,G_2),(E,G))$. 
The result follows when we rewrite the left side of \eqref{eq:xI2} in order to see that $\psi(x)$ is of the form $x\circ f$ for some $f\in\Hom((E_1,G_1),(E_2,G_2))$.  Define the lattice
\[J:= \frac{1}{\deg\varphi_{I_2}}\widehat{\varphi}_{I_2}\Hom((E_1,G_1),(E_2,G_2))\varphi_{I_1}\]
in $B_{p,\infty}^E$. 
By Lemma~\ref{lem:v_42.2.22}, 
\[\Hom((E_2,G_2),(E,G))\Hom((E_1,G_1),(E_2,G_2)) = \Hom((E_1,G_1),(E,G))\]
and so $I_2J = I_1$, as lattices in $B_{p,\infty}^E$. The right order of $I_2$ and the left order of $J$ are both the ring of endomorphisms of $(E_2,G_2)$ (namely the order $\frac{1}{\varphi_{I_2}}\widehat{\varphi}_{I_2}\OO(E_2,G_2)\varphi_{I_2}$ in $B_{p,\infty}^E$). Thus, $J = I_2^{-1}I_1$. Since $\beta\in I_2^{-1}I_2=J$, $\beta$ must be of the form
\[\beta=\frac{1}{\deg\varphi_{I_2}}\widehat{\varphi}_{I_2}\circ f\circ\varphi_{I_1},\]
for some $f\in\Hom((E_1,G_1),(E_2,G_2))$. Plugging this into Equation~\eqref{eq:xI2} above:
\begin{equation*}
\begin{split}
    \psi(x)\circ\varphi_{I_1} &= x\circ\varphi_{I_2}\circ\beta\\
    \psi(x)\circ\varphi_{I_1} &= x\circ f\circ\varphi_{I_1}\\
    \psi(x) &= x\circ f.
\end{split}
\end{equation*}
\end{proof}

\section{The Level Structure Graph}\label{sec:Graph}

For distinct fixed primes $p$ and $\ell$, and a fixed positive integer $N$ coprime to $p\ell$, we define the supersingular elliptic curves with level-$N$ structure $\ell$-isogeny graph.

\begin{definition}[Supersingular elliptic curves with level-$N$ structure $\ell$-isogeny graph, $\EpCat$]\label{def:EpCat_GraphDef}
In the graph $\EpCat$, vertices are $\overline{\mathbb{F}}_p$-isomorphism classes of pairs $(E,G)$, where $E$ is a supersingular elliptic curve and $G$ is a cyclic subgroup of $E[N]$ of order $N$. An edge from vertex $(E,G)$ to vertex $(E',G')$ is a degree-$\ell$ isogeny $\varphi:E\to E'$ such that $\varphi(G) = G'$. 
\end{definition}
The objects of $\EGCat$ form the nodes of the graph $\EpCat$. If we restrict the morphisms of $\EGCat$ to isogenies of degree $\ell$, we have the set of edges of $\EpCat$.

The graph structure is easily described in the special case where $N$ is prime: For each supersingular elliptic curve $E/\Fpbar$ with $j(E)\neq0,1728$, there are $N+1$ vertices of $\EpCat$. For $E/\Fpbar$ with $j(E) = 0$ or $1728$, there are at most $N+1$ vertices of $\EpCat$: the extra automorphisms of these $j$-invariants may map order-$N$ subgroups to each other (see Section~\ref{ssec:pnot1mod12}). There is a map of graphs from $\EpCat$ to $\FpbarGraph$ which is $(N+1)$-to-$1$ on vertices away from $j=0,1728$. For any prime $\ell$ coprime to $pN$, a supersingular elliptic curve over $\Fpbar$ has precisely $\ell+1$ degree-$\ell$ isogenies. Each edge corresponds to an isogeny \textit{up to post-composition with a curve automorphism}. If $j(E)\neq0,1728$, the automorphism group $\Aut(E) = [\pm1]$. Both automorphisms $[\pm1]$ act as the identity on the groups defining kernels. As a result, the duals of distinct isogenies must give distinct arrows in the graph. The graph can be taken to be undirected by identifying isogenies with their duals. See Figure~\ref{fig:p37N3ell2}.

If $j(E)=0$ or 1728, the automorphism groups expand to $\Aut(E_0) = \{[\pm1],[\pm\zeta_3],[\pm\zeta_3^2]\}$ and $\Aut(E_{1728})=\{[\pm 1], [\pm i]\}$. The `extra' automorphisms potentially swap kernels, meaning that the duals of distinct isogenies need not give distinct arrows in the graph. In this case, we do not draw the edges of the graph as undirected.

\begin{example}[$p = 37$, $N = 3$, $\ell=2$]\label{examp} We provide a reference example of the graph $\EpCat$ in Figure~\ref{fig:p37N3ell2}. As $p=37\equiv 1\pmod {12}$, this graph is drawn undirected by associating isogenies with their duals. Let $\mathbb{F}_{37}[s]/(s^2 + 33s + 2)$. The vertices are labeled with ordered pairs, the first element denoting the isomorphism class of elliptic curves with $j$-invariant $j$ by $E_j$. Let $\alpha:=10s + 20$, $\overline{\alpha} = 27s+23$ denote the $j$-invariants defined over $\mathbb{F}_{37^2}\setminus\mathbb{F}_{37}$. The supersingular elliptic curves over $\overline{\mathbb{F}}_{37}$ have 3-torsion defined over $\mathbb{F}_{37^4}:= \mathbb{F}_{37}[a]/(a^4 + 6a^2 + 24a + 2)$. We denote the 3-torsion subgroups using the $a^3$ term of the $x$-coordinate of a generating point, as computed in Sage \cite{sage}. The vertex appearance (shading and line style) aligns with the corresponding quaternion vertex, seen in Figure~\ref{fig:p37N3ell2Quat}. The corresponding supersingular $2$-isogeny graph is shown in Figure~\ref{fig:p37ell2}.
\end{example}

\begin{theorem}[Connectedness of $\EpCat$]\label{thm:Connectedness}
The graph $\EpCat$ consists of one connected component, for any pairwise coprime choices of $p,N,\ell$.
\end{theorem}
\begin{proof}
The connectedness of the graph follows from the work of Goren--Kassaei. In \cite{goren_kassaei}, the authors consider the $\ell$-isogeny graph with level-$N$ structure given by a choice of $N$-torsion point.  The connectedness of $\EpCat$ follows, as there is a map from the Goren-Kassaei graph to $\EpCat$ that acts surjectively on the vertex sets.
\end{proof}

Additionally, the result of Theorem~\ref{thm:Connectedness} can be seen as a corollary of a result provided by Roda \cite{RodaThesis}. Roda studies a supersingular $\ell$-isogeny level-$N$ structure graph whose vertices are pairs $(E,\alpha)$, where $\alpha:(\ZZ/N\ZZ)^2\xrightarrow{\sim} E[N]$.  In Section 3.3, Roda describes a means of counting the number of connected components of this graph. Choosing particular lifts of pairs $(E,G_1)$, $(E,G_2)$ for $G_1\neq G_2$, and showing that those lifts are connected using the conditions of \cite[Section 3.3]{RodaThesis}, we can prove that all of the points corresponding to a particular supersingular elliptic curve with level structure are connected in $\EpCat$. Together with the fact that the supersingular $\ell$-isogeny graph is connected, this proves that $\EpCat$ is connected as well.

\newpage
\begin{figure}
     \centering
         \begin{tikzpicture}
\node[fill=gray!40,rectangle,draw,thick] (E817) at (0,4) {$(E_8,17a^3)$};
\node[fill=gray!40,rectangle,draw,thick] (E820) at (4,4) {$(E_8,20a^3)$};
\node[thick,rectangle,draw] (E835) at (8,4) {$(E_8,35a^3)$};
\node[thick,rectangle,draw] (E82) at (12,4) {$(E_8,2a^3)$};

\node[fill=gray!10,dashed,rectangle,draw] (E10s4) at (0,2) {$(E_\alpha, 4a^3)$};
\node[dotted,rectangle,draw] (E10s16) at (4,2) {$(E_\alpha, 16a^3)$};
\node[dashed,rectangle,draw] (E10s23) at (8,2) {$(E_\alpha, 23a^3)$};
\node[thick,rectangle,draw] (E10s31) at (12,2) {$(E_\alpha, 31a^3)$};

\node[dotted, rectangle,draw] (E27s21) at (0,0) {$(E_{\overline{\alpha}},21a^3)$};
\node[fill=gray!10,dashed,rectangle,draw] (E27s33) at (4,0) {$(E_{\overline{\alpha}},33a^3)$};
\node[thick,rectangle,draw] (E27s6) at (8,0) {$(E_{\overline{\alpha}},6a^3)$};
\node[dashed,rectangle,draw] (E27s14) at (12,0) {$(E_{\overline{\alpha}},14a^3)$};

\draw[-, loop above] (E817) to (E817);
\draw[-] (E817) to (E10s4);

\draw[-,loop above] (E820) to (E820);
\draw[-] (E820) to (E10s16);

\draw[-] (E835) to (E10s23);
\draw[-] (E835) to (E82);

\draw[-] (E82) to (E10s31);

\draw[-] (E10s4.south) to [out = 330,in=150] (E27s14);
\draw[-] (E10s4) to (E27s33);

\draw[-] (E27s21) to (E10s16);
\draw[-] (E27s21.north) to (E10s16.west);
\draw[-] (E27s21.east) to [out = 60,in=300] (E817.east);

\draw[-] (E10s23) to (E27s6);

\draw[-] (E10s31) to (E27s6);
\draw[-] (E10s31) to (E27s14);

\draw[-] (E27s33.east) to [out = 60,in = 300] (E820.east);
\draw[-] (E27s33.north) to (E10s23.west);

\draw[-] (E27s6.east) to [out = 60,in = 300] (E835.east);

\draw[-] (E27s14.east) to [out = 60,in = 300] (E82.east);
\end{tikzpicture}
    \caption{Graph of $\mathcal{E}_{37,2}^3$, with groups labeled by the first term in the $x$-coordinate of a generating point.}
    \label{fig:p37N3ell2}
\end{figure}

\begin{figure}[]
    \centering
    \begin{subfigure}{.25\textwidth}
    \centering
             \begin{tikzpicture}
\node[rectangle,draw] (E8) at (-2,4) {$E_8$};
\node[rectangle,draw] (Ealpha) at (-2,2) {$E_\alpha$};
\node[rectangle,draw] (Ealphabar) at (-2,0) {$E_{\overline{\alpha}}$};
\draw[-,loop above] (E8) to (E8);
\draw[-] (E8) to (Ealpha);
\draw[-] (E8.west) to [out = 210,in=150] (Ealphabar.west);
\draw[-] (Ealpha) to (Ealphabar);
\draw[-] (Ealpha) to [out = 300,in=60] (Ealphabar);
\end{tikzpicture}
    \caption{Supersingular \\2-isogeny graph over $\overline{\mathbb{F}}_{37}$}\label{fig:p37ell2}
    \end{subfigure}
    \begin{subfigure}{.74\textwidth}
    \centering
        \begin{tikzpicture}
\node[thick,circle,draw] (O0) at (1,2) {$\OO_0$};
\node[dashed,circle,draw] (O1) at (3,2) {$\OO_1$};
\node[fill=gray!10,dashed,circle,draw] (O2) at (5,2) {$\OO_2$};
\node[dotted,circle,draw] (O3) at (2,0) {$\OO_3$};
\node[fill=gray!40,circle,draw,thick] (O4) at (4,0) {$\OO_4$};

\draw[->,loop left] (O0) to (O0);
\draw[->,loop above] (O0) to (O0);
\draw[->] (O0) to (O1);

\draw[->] (O1) to [out=150,in=30] (O0);
\draw[->] (O1) to [out=210,in=330] (O0);
\draw[->] (O1) to (O2);

\draw[->] (O2) to [out=150,in=30] (O1);
\draw[->] (O2) to (O4);
\draw[->, loop above] (O2) to (O2);

\draw[->] (O4.east) to [out=0,in=270] (O2.south);
\draw[->, loop below] (O4) to (O4);
\draw[->] (O4) to (O3);

\draw[->] (O3) to [out=30,in=150] (O4);
\draw[->,loop below] (O3) to (O3);
\draw[->,loop left] (O3) to (O3);

\end{tikzpicture}
\caption{Graph of level-3 Eichler orders in $B_{37,\infty}$ with connecting ideals of norm 2. }\label{fig:p37N3ell2Quat}
    \end{subfigure}
\caption{}
\end{figure}

\bibliographystyle{plain}
\bibliography{biblio}\vspace{0.75in}

\end{document}